\documentclass[12pt]{article}% Math and Physical Sciences Reference Style

\usepackage{authblk} % \affil
\usepackage{hyperref} % ref in PDF
\usepackage{dsfont}
\usepackage[T1]{fontenc} 
\usepackage{amsmath}
\usepackage{amsfonts,amssymb}
\usepackage{extarrows}
\usepackage{geometry}
\usepackage{enumerate}
\usepackage{enumitem}

\usepackage[utf8]{inputenc}
\def\0{\emptyset}

\newtheorem{theorem}{Theorem}[section]
\newtheorem{definition}[theorem]{Definition}
\newtheorem{lemma}[theorem]{Lemma}

\newtheorem{cor}[theorem]{Corollary}
\newtheorem{prop}[theorem]{Proposition}

\newenvironment{proof}{{\noindent\it Proof.}}{\hfill $\square$\par}
\usepackage{graphicx}

\usepackage{enumerate}

\usepackage{
	amsmath,			% Math Environments
	amssymb,			% Extended Symbols
	enumerate,		    % Enumerate Environments
	graphicx,			% Include Images
	lastpage,			% R\part{title}eference Lastpage
	multicol,			% Use Multi-columns
	multirow,			% Use Multi-rows
	pifont,			    % For Checkmarks
}

\usepackage[numbers]{natbib}
\begin{document}

% --- PAPER INFO ---

\title{Algorithm for finding vertex-edge domination number on graphs with bounded treewidth and related problems on planar graphs}

\author[1]{\small\bf Yichen Wang\thanks{E-mail: wangyich22@mails.tsinghua.edu.cn}}
\author[1]{\small\bf Haixiang Zhang\thanks{E-mail:  zhang-hx22@mails.tsinghua.edu.cn}}
\author[2]{\small\bf Mei Lu\thanks{E-mail: lumei@tsinghua.edu.cn}}

\affil[1]{\small Department of Mathematical Sciences, Tsinghua University, Beijing, P.R. China.}

%\author{Yichen Wang}
\date{}

\maketitle

\abstract{
Given a graph $G=(V,E)$, a vertex $u \in V$ {\em ve-dominates} all edges incident to any
vertex of $N_G[u]$. A set $S \subseteq V$ is a {\em ve-dominating set} if for all edges $e\in E$, there exists a vertex $u\in S$ such that $u$ ve-dominates $e$.
The minimum cardinality among all ve-dominating sets is known as the \textit{vertex-edge domination number} (or simply ve-domination number) and denoted by $\gamma_{ve}(G)$.
Finding a minimum ve-dominating set was proved to be NP-complete. Restricted to trees, the problem admits a linear-time algorithm. Treewidth is a commonly used parameter for solving NP-hard problems. In this paper, we present a polynomial-time algorithm for finding a minimum ve-dominating set on graphs with bounded treewidth.
Moreover, we show that the treewidth of a planar graph $G$ with ve-domination number $\gamma_{ve}(G)$ is $O(\sqrt{\gamma_{ve}(G)})$ and present an $O(c^{\sqrt{k}}|V(G)|)$-time algorithm for the $k$-ve-domination problem on planar graphs.
}

\vskip.3cm

\section{Introduction}\label{sec1}

Let $G=(V,E)$ be a simple undirected graph.
Given a vertex $u \in V$, let $N_G(u)$ be the open neighbor set of $u$, that is, $N_G(u) = \{ v \in V(G)~|~ uv \in E(G)\}$.
Let $N_G[u] = N_G(u) \cup \{u\}$ be the closed neighbor set of $u$ in $G$.
For a vertex set $S$, let $N_G(S) =\bigcup_{v\in S}N_G(v) - S $ and $N[S] = N(S) \cup S$ be the open neighbor set and closed neighbor set of $S$, respectively for a subset $S \subset V(G)$.
When there is no ambiguity, we ignore the subscript $G$ in the above notation.
For an edge $e \in E$, we also use $e$ to denote the subset consisting of two vertices of $e$ for short.
An edge $e \in E$ is \textit{vertex-edge dominated} (or simply ve-dominated) by a vertex $u \in V$ if $e \cap N[u] \neq \emptyset$. A vertex set $S \subseteq V$ is a \textit{vertex-edge dominating set} (or simply ve-dominating set) if for all edges $e \in E$, there exists a vertex $u \in S$ such that $u$ ve-dominates $e$.
The minimum cardinality among all ve-dominating sets is known as the \textit{vertex-edge domination number} (or simply ve-domination number) and denoted by $\gamma_{ve}(G)$. A $k$-dominating set $D$ of $G $ is a set of $k$ vertices of $G$ such that each of the rest of the vertices has at least one neighbor in $D$. The minimum $k$ such that $ G$ has a $k$-dominating set is called the domination number of $G$, denoted by $\gamma (G)$. Obviously, $\gamma_{ve}(G)\le \gamma (G)$.
%If we give each vertex a weight, we nartually have a weighted version of ve-dominating set.

The vertex-edge domination was first introduced by Peters~\cite{vedomination_introduce} and received more attention after Lewis~\cite{vedomination_introduce2} established many new results. Lewis~\cite{vedomination_introduce2} gave lower bounds on $\gamma_{ve}(G)$ for different classes of graphs, such as connected graphs, $k$-regular graphs, cubic graphs, etc. Krishnakumari, Venkatakrishnan and Krzywkowski~\cite{krishnakumari2014bounds} gave the upper and lower bounds of ve-domination number on trees. For other structural results, the readers are referred to~\cite{boutrig2016vertex, zylinski2019vertex}.

From the algorithmic side, Lewis~\cite{vedomination_introduce2} proved that the ve-domination problem is NP-complete for bipartite, chordal, planar, and circle graphs. Lewis~\cite{vedomination_introduce2} also proposed a linear-time algorithm for finding a minimum ve-dominating set on trees. However, Paul and Ranjan~\cite{paul2022} proved that Lewis's algorithm is flawed and proposed a new linear-time algorithm for finding a minimum ve-dominating set on block graphs and a new linear-time algorithm for finding a weighted minimum ve-dominating set on trees. Paul and Ranjan~\cite{paul2022} also proved that finding a minimum ve-dominating set is NP-complete for undirected path graphs.

Decompositions play an important role in the graph theory.
Various decompositions of graphs such as decomposition by clique separators, tree-decomposition and clique-decomposition are often used to design efficient graph algorithms. 
In this paper, we consider the ve-domination problem with the commonly used parameter treewidth.  
Treewidth is a parameter that plays a fundamental role in various graph algorithms.
The treewidth of a graph gives an indication of how far away the graph is from being a tree or forest.
The closer the graph is to being a forest, the smaller is its treewidth.
It is well-known that many NP-complete problems can be solved in polynomial time on graphs of bounded treewidth.
A variety of problems on graphs can be solved in linear time for graphs with bounded treewidth including many domination-like problems \cite{treewidth_sample1_Arnborg1991}, like dominating set \cite{dominationset_treewidth_alber2002}, Roman dominating set~\cite{peng2007roman} and vertex cover $P_3$ problem \cite{VCP3_treewidth_bai2019}.

% The ve-domination is an variation of dominating set.
For dominating set, it was shown that the treewidth of a planar graph $G$ with domination number $\gamma(G)$ is $O(\sqrt{\gamma(G)})$~\cite{ABo} and this fact is used as the basis for several fixed parameter algorithms on planar graphs.
In this paper, we will show that such a relationship remains between treewidth and the ve-domination number, that is, $tw(G) = O(\sqrt{\gamma_{ve}(G)})$. 
Consider a graph $G$ whose vertex set is $\{x, y_i, z_i\}_{1 \le i \le n}$ and edge set is $\{xy_i, y_iz_i\}_{1\le i \le n}$.
It is easy to prove that $\gamma(G) = n $ and $\gamma_{ve}(G) = 1$. Then $\lim \limits_{n \rightarrow +\infty} \gamma(G) / \gamma_{ve}(G) = +\infty$.
In this example, $\gamma(G)$ is ``asymptotically strictly larger'' than $\gamma_{ve}(G)$ and thus, our result is ``strictly better'' than the previous results.

The rest of the paper is organized as follows.
In Section~\ref{sec: Preliminary}, we introduce some notation and basic knowledge about treewidth.
In Section~\ref{sec: tree width algorithm}, we propose a linear-time algorithm to find a minimum ve-dominating set on graphs with bounded treewidth.
In Section~\ref{sec: planar treewidth ve-domination number}, we show that the treewidth of a planar graph $G$ is $O(\sqrt{\gamma_{ve}(G)})$.
In Section~\ref{sec: planar graph algorithm}, we present an algorithm solving $k$-ve-domination problem on planar graphs in $O(c^{\sqrt{k}}n)$ time where $n$ is the order of  $G$.

%In Section~\ref{sec: conclusions}, we discuss the relationship between split-width and treewidth.

\section{Preliminary}\label{sec: Preliminary}

For a graph $G$ and a vertex set $S \subseteq V(G)$, let $G[S]$ be the subgraph of $G$ induced by $S$.
We first give an equivalent definition of ve-dominating set.

%\vspace{.2cm}
\begin{prop}\label{prop: equi def of ve-domination}
	For a graph $G$, a vertex set $D$ is a ve-dominating set of $G$ if and only if $V(G) - N[D]$ is an independent set.
\end{prop}

\begin{proof}
	Let $D$ be a ve-dominating set of $G$. If there exist $x,y\in V(G)-N[D]$ such that $xy \in E(G)$, then  $\{x,y\}\cap N[D]=\emptyset$, a contradiction.

	Suppose $D$ is a vertex set satisfying that $V(G) - N[D]$ is an independent set. Then for any edge $xy \in E(G)$, $\{x,y\}\cap N[D]\not=\emptyset$ which implies
	 $D$ is a ve-dominating set.
\end{proof}
% The proof is simple.

In the following, we may understand a ve-dominating set from both perspectives.
That is, a ve-dominating set $D$ of $G$ dominates every edge and satisfies that $V(G) - N[D]$ is an independent set.
Given a ve-dominating set $D$ of $G$, a vertex $x \in V(G)$ is \textit{dominated} by $D$ if $x \in N[D]$ and $x$ is \textit{abandoned} by $D$ if $x \in V(G) - N[D]$.

Now we give definitions related to treewidth.  The treewidth of a graph is defined through
the concept of tree-decompositions. 
% In the following, we will use $I$ to denote the vertex set of $T$ when $T$ is a tree and we call the vertex of $T$ {\em node}.
\vskip.2cm
\begin{definition}\label{def: tree-decomposition}
	A tree-decomposition of a graph $G = (V; E)$ is a pair $(X; T)$, where $T(I; F)$ is a tree with vertex set $I$ and edge set $F$, and $X = \{X_i \mid i \in I\}$ is a family of subsets of $V$, one for each node of $T$, such that:
	\begin{itemize}
		\item $\bigcup \limits_{i \in I}X_i = V$;
		\item for each edge $(u,v) \in E$, there exists an $i \in I$ such that $u,v \in X_i$;
		\item for all $i,j,k \in I$, if $j$ is on the path from $i$ to $k$ in $T$, then $X_i \cap X_k \subseteq X_j$.
	\end{itemize}
	Here, the vertices in $T$ are called \textit{nodes}.
	The sets $\{X_i\}$ are called \textit{bags}.
\end{definition}

The  width of a tree-decomposition $(X; T)$ is $\max \limits_{i\in I} |X_i| - 1$.
The treewidth  of $G$    is the minimum treewidth over all possible tree-decompositions of $G$.
We need the definition of nice tree-decomposition~\cite{ABo}.
\vskip.2cm

\begin{definition}
A tree-decomposition $(X; T)$ of  $G$ is \textit{nice}, if it satisfies the following properties:
\begin{itemize}
	\item Every node of $T$ has at most  two child nodes.
\vskip.2cm
	\item If a node $i$ has two child nodes $j$ and $k$, then $X_i = X_j = X_k$ and $i$ is called a {\em join node}.
\vskip.2cm
	\item If a node $i$ has one child $j$, then one of the following situations must hold:
		\begin{enumerate}[leftmargin=2em, label= (\alph*)]
			\vskip.2cm
			\item $|X_i| = |X_j| + 1$ and $X_j \subset X_i$, $i$ is called an {\em introduce node}, or
			\vskip.2cm
			\item $|X_i| = |X_j| - 1$ and $X_i \subset X_j$, $i$ is called a {\em forget node}.
		\end{enumerate}
\end{itemize}
\end{definition}

It can be shown that any tree-decomposition of  $G$ can be transformed into a nice tree-decomposition of $G$ with the same treewidth and size $O(n)$ in linear time \cite{nice_tree_decomposition_kloks_book1994}.

%Let $(X;T)$ be a tree-decomposition of  $G(V,E)$ and  $i$ be a node in $T$.
%We define $T_i$ to be the subtree of $T$ rooted at $i$, including $i$. Let $X_i$ be the corresponding bag of $i$.
A {\em rooted tree-decomposition} is a tree-decomposition with a distinguished root
node, denoted by $r$. Given a rooted tree-decomposition $(X;T)$ with a root node $r$ and a node $i$ of $T$, let $Desc(i)$
be the set of descendants of node $i$ in $T$, including $i$; let $T_i = T [Desc(i)]$ be a subtree of $T$ rooted at $i$; let $G_i =
G[V_i]$, where $V_i=\bigcup_{s\in T_i}X_s$. That is, $G_i$ is induced by the vertices in the bags of subtree $T_i$. Then $T_r=T$ and $G_r=G$.
%If $i$ is a leaf, we define $V_i = X_i, E_i = \{uv\in E ~|~ u,v\in X_i\}$; if $i$ is an internal node with children $L,R$, we define $V_i = V_L \cup V_R, E_i = E_L \cup E_R$.
%The graph $G(V_i,E_i)$ is denoted by $G_i$.

% ------------------------------- split width algorithm -----------------------

% ---------------------- Tree width algorithm --------------------------
\section{An algorithm for ve-domination number} \label{sec: tree width algorithm}
In this section, we  describe an algorithm to find a minimum ve-dominating set on graphs with bounded treewidth. Let $T(I; F)$ be a nice tree-decomposition of $G$ rooted at $r$. We first give some results on $T(I; F)$.

\vskip.2cm
\begin{prop}\label{prop: edge properties in a tree-decomposition for introduce node} Let $i$ be an introduce node and  $j$  its child node.
	Let $X_i - X_j = \{y\}$.
	For every node $k$ in $T_j$ but $k \neq j$, if $x \in X_k-X_i$, then $xy \notin E$.
\end{prop}
\vspace{0.3em}
\noindent{\bf Proof.}
Suppose $xy \in E$. By the definition of tree-decomposition, there exists a node $t$ such that $x,y \in X_t$ and $t$ is a node in $T_i$.
Since $x \notin X_i$ and $y\notin X_j$, we have $t$ is a node in $T_j$ and $t \not=j$.
 Thus we have $y \in X_j$ by the definition of tree-decomposition, a contradiction. $\square$

\vspace{0.3em}

By Proposition \ref{prop: edge properties in a tree-decomposition for introduce node}, the following result is obvious.
\vspace{0.3em}
\begin{cor}\label{coro: properties of N(y) for introduce node}
	Let $i$ be an introduce node and  $j$ its child node. Suppose $X_i - X_j = \{y\}$.
	Then $N_{G_i}(y) \subseteq X_j$.
\end{cor}
\vspace{0.5em}

\begin{prop}\label{prop: edge properties in a tree-decomposition for join node}
Let $i$ be a join  node and  $j,k$ be its two child nodes.	
	If $x \in V_j - X_i$ and $y \in V_k - X_i$,
	then $xy \notin E$.
\end{prop}

\noindent{\bf Proof.}
Suppose $xy \in E$. By the definition of tree-decomposition, there exists a node $t$ such that $x,y \in X_t$.
Since $x,y \notin X_i$, we have that $t$ must be a node in $T_i$ and
$t \notin\{ i, j, k\}$.
Then $t$ is a node in $T_j$ or $T_k$, say $T_j$. Thus we have $y \in X_i$ by the third condition of Definition~\ref{def: tree-decomposition}, a  contradiction.
$\square$
\vspace{0.5em}

The following result is obvious by Proposition~\ref{prop: edge properties in a tree-decomposition for join node}.

\vspace{0.5em}
\begin{cor}\label{coro: properties of edge connectivity in join node}
Let $i$ be a join  node and  $j,k$ its two child nodes.	
	For every vertex $x \in V_j - X_i$, we have $N_{G_i}(x) \subseteq V_j$ and $N_{G_j}(x) = N_{G_i}(x)$.
\end{cor}

\vspace{0.5em}
Now we present the main theorem of this paper.
The algorithm is designed by the idea of dynamic programming.

\vspace{0.5em}
\begin{theorem}\label{thm: main thm: poly algorithm of ve-dominating set}
	For every fixed integer $k$, there is a deterministic algorithm that, given a graph $G$ with treewidth at most $k$ on $n$ vertices,
	calculates the ve-domination number of $G$ and finds such a ve-dominating set in time $O(n5^{k+1})$.
\end{theorem}
\vspace{0.5em}

\noindent\textbf{Proof of Theorem~\ref{thm: main thm: poly algorithm of ve-dominating set}:}
Let $T(I; F)$ be a nice tree-decomposition of $G$ rooted at $r$ with width $k$. Then $|X_i|:=n_i\le k+1$ for all $i\in I$. In the following,
 four
digits will be assigned to the vertices in the bag $X_i$:
\begin{enumerate}[label=``\arabic*'']
	\item  meaning that the vertex belongs to the ve-dominating set,
	\item  meaning that the vertex is already dominated at the current stage of the algorithm,
	\item meaning that the vertex is still asking for a domination at the current stage of the algorithm, and
	\item meaning that the vertex will be abandoned.
\end{enumerate}

We hope that, for a ve-dominating set $S$ of $G$, the vertices in $S$ are labelled $1$, the vertices in $N(S)\setminus S$ are labelled $2$, and the vertices in $V - N[S]$ are labelled $4$.
Since we will proceed the algorithm from the leaves to the root, at the stage of $X_i$, the vertices of label $2$ are dominated by some vertex in $V_i$ and vertices of label $3$ should be dominated by some vertex in $S\setminus V_i$ and not be dominated by any vertex in $S\cap V_i$.

For each bag $X_i$, the \textit{guessing}  of $X_i=\{x_1,\ldots,x_{n_i}\}$ is a mapping $g:X_i\rightarrow \{1,2,3,4\}$ assigning four different digits  to the vertices in the bag. Then there are at most $4^{k+1}$ guessings of $X_i$.
Let $\mathcal{G} (X_i)$ be the set of all guessings on $X_i$.
For $d=1,2,3,4$ and $g \in \mathcal{G}(X_i)$, let $Z_d(g)$ be the set of vertices in $X_i$ with label $d$, that is,
\begin{equation}
	Z_{d}(g) = \{x \in X_i ~|~ g(x) = d\}, d=1,2,3,4.
\end{equation}

\begin{definition}\label{def: partial ve-dominating set}
	Let $g \in \mathcal{G}(X_i)$, a vertex set $S_i \subseteq V_i$ is a \textit{partial ve-dominating set} of $g$ on $G_i$
if $S_i$ has the following properties:
\begin{enumerate}
	\item $S_i \cap X_i = Z_1(g)$ (meaning the selected vertices are properly guessed),
	\item $Z_2(g) \subseteq N_{G_i}(S_i)$ (meaning the vertices of label $2$ are dominated),
	\item $Z_3(g) \cap N[S_i] = \emptyset$ (meaning the vertices of label $3$ are not dominated by any vertex in $S_i$),
	\item $Z_4(g) \cap N[S_i] = \emptyset$ (meaning the vertices of label $4$ are properly guessed),
	\item $V_i - N_{G_i}[S_i] - Z_3(g) $ is an independent set of $G_i$ (meaning $S_i$ is a ve-dominating set of $V_i - Z_3(g)$).
\end{enumerate}
\end{definition}
Given a guessing $g \in \mathcal{G}(X_i)$, let
\begin{equation*}
D(g) = \min \left( \{ |S_i| ~|~ \mbox{ $S_i \subseteq V_i$ is a partial ve-dominating set of $g$ on $G_i$} \} \cup \{ + \infty \} \right),
\end{equation*}
where $D(g)=+ \infty$ means there is no partial ve-dominating set of $g$. 

We say a guessing $g \in \mathcal{G}(X_i)$ is \textit{valid} if the following holds:
\begin{enumerate}
	\item there is no edge between $Z_1(g)$ and $Z_4(g)$,
	\item there is no edge between $Z_1(g)$ and $Z_3(g)$,
	\item $Z_4(g)$ is an independent set,
	% \item and $Z_2(g) \subseteq N(Z_1(g))$.
\end{enumerate}
It directly follows from the definition that $D(g) = +\infty$ if $g$ is not valid.
So in the following, we only consider $D(g)$ for valid guessings.
We call $D(g)$ \textit{the partial ve-domination number} of $g$. 
Notice that for $g \in \mathcal{G}(X_r)$, if $Z_3(g) = \emptyset$, then a partial ve-dominating set of $g$ is also a ve-dominating set of $G$.
That is, $D(g)$ is the ve-domination number in this case.

% During the algorithm, we will try to enumerate all ``good'' guessings at each node from leaves to the root.
% Finally, at the root node, we will output the ``good'' guessing with the minimum partial ve-domination number with $Z_3(g) = \emptyset$.

% and the partial ve-domination number is equal to ve-domination number.
% FIGURE
%\begin{figure}[t]\label{fig: partial ve-dominating set}
    %\centering
   % \includegraphics[width=\linewidth]{./src/PartialVEDominatingSet.jpg}
    %\caption{ Relationship of sets of a partial ve-dominating set }
%\end{figure}

%Notice that the vertices labeled $3$ is not considered because they will be dominated by hyper-layer vertices.
%The minimum cardinality among all partial ve-dominating set $S_i$ of guessing $g$ is called \textit{the partial ve-domination number} of $g$ denoted by $D(g)$.
%If there is no partial ve-dominating set of given guessing $g$, then the partial ve-domination number is denoted as $+\infty$.
%That is,

% During the algorithm, we need to calculate the size of the minimum partial ve-dominating set of each guessing.
% For each node $i$, we will use a mapping $F_i: \mathcal{G} (X_i) \rightarrow \mathbb{N} \cup \{+\infty \}$. For $g\in \mathcal{G}(X_i)$,  the value $F_i(g)$  records in the algorithm  the partial ve-domination number of $g$.
% For each type of nodes in the nice tree-decomposition, we will propose an algorithm to calculate $F_i$ and show that for each  node $i$ and $g \in \mathcal{G}(X_i)$, $F_i(g) = D(g) $.
%During the calculation of $F_i$, we may use the value of $\{F_j ~|~ \mbox{$X_j$ is the child of $X_i$}\}$.

\noindent{\bf Step 1. } In the first step of the algorithm, for each leaf node $i$ of the tree-decomposition,
we initialize the $D(g)$ for all $g \in \mathcal{G}(X_i)$. Let $i$ be a leaf node. 
% A guessing $g \in \mathcal{G}(X_i)$ is \textit{canonical} if for all $x \in Z_2(g)$, $N(x) \cap Z_1(g) \neq \emptyset$.
For any valid $g \in \mathcal{G}(X_i)$, it is easy to see that 
\[
D(g) = \left\{\begin{aligned}
	&|Z_1(g)| &\qquad& \text{ if $g$ is valid}, \\
	&+\infty &\qquad& \text{ otherwise. }
\end{aligned}
\right.
\]
% \[D(g) = \left\{\begin{aligned}
% &|Z_1(g)| &\qquad& \text{ if } g \text{ is canonical and valid}, \\
% &+\infty &\qquad& \text{ otherwise. }
% \end{aligned}
% \right.\]	
In the first case, $Z_1(g)$ is the only partial ve-dominating set of $g$ on $G_i$.
In the second case, there is no partial ve-dominating set of $g$.
 This step can be carried out in at most $O(4^{k+1}k^2)$ time for each leaf node $i$, since there are at most $4^{k+1}$ guessings on $X_i$ and to check whether a guessing is valid takes time $O(k^2)$.
% In the following, we will characterize the function of $D(g)$ for each type of nodes in the tree-decomposition.
% Trivially, the mappings of the leaves are monotonous, yielding the induction base.

\noindent{\bf Step 2. } Now we visit the nodes of the tree-decomposition from the leaves to the root, evaluating $D(g)$ for each valid guessing according to the following rules.

\textbf{Forget node:}
Suppose $i $ is a forget node and let $j$ be its child. Assume $X_j = X_i \cup \{y\}$.
Let $g \in \mathcal{G}(X_i)$. A guessing $g' \in \mathcal{G}(X_j)$ is \textit{compatible} with $g$ if the following conditions hold.
\begin{enumerate}
	\item For all $x \in X_i$, $g'(x) = g(x)$.
	\item $g'(y) \in \{1, 2, 4\}$.
\end{enumerate}
It is natural to understand the first rule. 
The second rule actually says that the vertex $y$ cannot be labelled $3$, because $y$ is going to be forgotten.
If it is dominated, it must be dominated by some vertex in $V_j$.

For any valid $g \in \mathcal{G}(X_i)$, we claim that
% \[F_i(g) = \left\{\begin{aligned}
% &\min \{F_j(g') ~|~ g' \in \mathcal{G}(X_j) \text{ and } g' \text{  is compatible with } g \} &\qquad& \text{ if } g \text{ is valid} , \\
% &+\infty &\qquad& \text{ otherwise.}
% \end{aligned}
% \right.\]
\begin{equation}\label{eq: forget node}
	D(g) = \min \{D(g') ~|~ g' \in \mathcal{G}(X_j) \text{ and } g' \text{  is compatible with } g \}.
\end{equation}
% We show that for each forget node $i$ and $g \in \mathcal{G}(X_i)$, $F_i(g) = D(g) $.
If $D(g) < +\infty$, let $S_i$ be a partial ve-dominating set of $g$ on $G_i$ such that $D(g) = |S_i|$.
Let
\begin{equation}
	g'(z) = \left\{ \begin{array}{ll}
		g(z) & z \in X_j-\{y\},\\
		1 & z=y, y \in S_i,\\
		2 & z=y, y \in N(S_i), \\
		4 & z=y, y \in V_i - N[S_i].
	\end{array} \right.
\end{equation}
It is easy to verify that $S_i$ is also a partial ve-dominating set of $g'$ on $G_j$ and $g'$ is compatible with $g$. 
Therefore, $D(g) \ge \min \{ D(\tilde{g}) \mid \tilde{g} \in \mathcal{G}(X_j),\ \tilde{g} \text{ is compatible with } g \}$.
On the contrary, suppose that $M := \min \{ D(\tilde{g}) \mid \tilde{g} \in \mathcal{G}(X_j),\ \tilde{g} \text{ is compatible with } g \} < +\infty$, and let $g' \in \mathcal{G}(X_j)$ be compatible with $g$ such that $D(g') = M$.
Then let $S_j$ be a partial ve-dominating set of $g'$ on $G_j$ such that $D(g') = |S_j|$ and let $S$ be a ve-dominating set of $G$ such that $S_j = S \cap V_j$.
Note that $V_i = V_j$, then it is easy to verify that $S_j$ is also a partial ve-dominating set of $g$ on $G_i$.
Then $D(g') = M \le D(g)$.
As a conclusion, we have proved (\ref{eq: forget node}).

Given  $g\in \mathcal{G}(X_i)$, there are at most three guessings $g'\in \mathcal{G}(X_j)$ compatible with $g$. 
Note that there are at most $4^{k+1}$ guessings on $X_i$.
Determine whether a guessing is valid takes time $O(k^2)$.
Then to calculate $D(g)$ for all $g \in \mathcal{G}(X_i)$ for a forget node takes time $O(4^{k+1}k^2)$.

\textbf{Introduce node:}
Suppose $i $ is an introduce node, and let $j$ be its child. Assume $X_j = X_i - \{y\}$. Then $V_j = V_i - \{y\}$.
For any $x\in X_i - \{y\}$ and $g \in \mathcal{G}(X_i)$, let $\phi(g) \subseteq \mathcal{G}(X_j)$ be a collection of guessings on $X_j$ such that for every $g' \in \phi(g)$, for each vertex $x \in X_j$, $g'(x) \in \{2,3\}$ if $g(x) = 2$ and $x \in N(y)$, and $g'(x) = g(x)$ otherwise.

For any  $g \in \mathcal{G}(X_i)$,
we use $g|_{X_j}$ to represent the guessing in $\mathcal{G}(X_j)$ by limiting $g$ to $X_j$.
For any valid $g \in \mathcal{G}(X_i)$, we claim that
\begin{equation}\label{eq: introduce node}
	D(g) = \left\{\begin{aligned}
	& \min_{g' \in \phi(g)} \{D(g') + 1\} &\qquad& \text{ if } g(y)=1, \\
	% &D(g |_{X_j}) &\qquad& \text{ if } g(y)=2 \text{  and } N(y) \cap Z_1(g) \neq \emptyset,\\
	&+\infty &\qquad& \text{ if } g(y)=2 \text{  and } N(y) \cap Z_1(g) = \emptyset,\\
	&D(g|_{X_j}) &\qquad& \text{ otherwise }.
	\end{aligned}
	\right.
\end{equation}

\vskip.2cm
{\bf Case 1:} $g(y) = 1$.
In this case, $y \in Z_1(g)$. 
First, we claim that when $D(g) < +\infty$, we have $\min_{g' \in \phi(g)} \{D(g') + 1\}  \le D(g)$.
Let $S_i$ be a partial ve-dominating set of $g$ on $G_i$ such that $D(g)=|S_i|$.
Then we would like to prove that $S_j \triangleq S_i \setminus \{y\}$ is a partial ve-dominating set on $G_j$ of some $g' \in \phi(g)$.
Define $g'$ as follows:
\begin{equation}
	g'(x) = \left\{ \begin{array}{ll}
		g(x) & \text{ if }g(x) \neq 2,\\
		2 & \text{ if }g(x) = 2, x \in N(S_j),\\
		3 & \text{ if }g(x) = 2, x \notin N(S_j).
	\end{array} \right.
\end{equation}
By the definition, we have $g' \in \phi(g)$.

By the definition of $\phi(g)$, it is easy to verify properties 1, 4 and 5 in Definition~\ref{def: partial ve-dominating set}, while properties 2 and 3 directly follow from the definition of $g'$.
Then $S_j$ is a partial ve-dominating set of $g'$ on $G_j$.
Thus $D(g) = |S_i| = |S_j| + 1 \ge \min_{g' \in \phi(g)} \{D(g') + 1\}$.

Then we claim that when $\min_{g' \in \phi(g)} \{D(g') + 1\} < + \infty$, let $g' \in \phi(g)$ be a guessing such that $D(g') + 1 = \min_{g' \in \phi(g)} \{D(g') + 1\}$, then $D(g) \le D(g') + 1$.
Let $S_j$ be a partial ve-dominating set of $g'$ on $G_j$ such that $D(g') = |S_j|$.
Then we would like to prove that $S_i \triangleq S_j \cup \{y\}$ is a partial ve-dominating set of $g$ on $G_i$.
Let us verify the properties in Definition~\ref{def: partial ve-dominating set} one by one.
\begin{enumerate}
	\item $S_i \cap X_i = S_j \cap X_j \cup \{y\} = Z_1(g') \cup \{y\} = Z_1(g)$.
	\item Let $x$ be a vertex in $Z_2(g)$.
	Note that $Z_2(g') \subseteq N_{G_j}(S_j) \subseteq N_{G_i}(S_i)$.
	If $x \notin Z_2(g')$, then $g'(x) = 3$.
	Since $g' \in \phi(g)$, $x \in N(y)$.
	Then in all cases, we have $x \in N_{G_i}(S_i)$, which leads to $Z_2(g) \subseteq N_{G_i}(S_i)$.
	\item $Z_3(g) \subseteq Z_3(g')$. Note that $Z_3(g') \cap N_{G_j}[S_j] = \emptyset$. Since $g$ is valid, $Z_3(g) \cap N(y) = \emptyset$. 
	Then $Z_3(g) \cap N_{G_i}[S_i] = \emptyset$.
	\item $Z_4(g) = Z_4(g')$. Similarly, $Z_4(g) \cap N_{G_i}[S_i] = \left( Z_4(g') \cap N_{G_j}[S_j] \right) \cup \left( Z_3(g) \cap N(y) \right) = \emptyset$.
	\item $V_i - N_{G_i}[S_i] - Z_3(g) = V_j - N_{G_j}[S_j] - Z_3(g')$ is an independent set.
\end{enumerate}
Then $D(g) \le |S_i| = D(g') + 1$.
As a conclusion, we have proved (\ref{eq: introduce node}) in this case.

\vskip.2cm

{\bf Case 2:} $g(y)=2$  and $N(y) \cap Z_1(g) = \emptyset$.
In this case, we would like to prove there is no partial ve-dominating set of $g$ on $G_i$.
Suppose otherwise, let $S_i$ be a partial ve-dominating set of $g$ on $G_i$.
Then since $y \in N(S_i)$, assume $z \in S_i$ is connected to $y$. 
By the condition, $z \notin X_i$.
By the second condition of Definition~\ref{def: tree-decomposition}, there exists a bag $X_k$ such that $y,z \in X_k$.
Since $z \notin X_i$, and $y$ is a new vertex introduced in $X_i$, then $k$ is not in the subtree $T_i$.
However, $z$ is a vertex in $V_i$ and $z$ is in a bag outside $T_i$, then by the third condition of Definition~\ref{def: tree-decomposition}, $z \in X_i$, a contradiction.

{\bf Case 3:} otherwise.

First, if $D(g) < +\infty$ we claim that $D(g|_{X_j}) \le D(g)$.
Let $S_i \subseteq V_i$ be a partial ve-dominating set of $g$ on $G_i$ such that $D(g)=|S_i|$.
Then we would like to prove that $S_j \triangleq S_i$ is a partial ve-dominating set of $g|_{X_j}$ on $G_j$ by verifying the properties in Definition~\ref{def: partial ve-dominating set} one by one.
\begin{enumerate}
	\item $S_j \cap X_j = S_i \cap X_i = Z_1(g) = Z_1(g|_{X_j})$.
	\item $Z_2(g|_{X_j}) \subseteq Z_2(g)\setminus \{y\} \subseteq N_{G_i}(S_i) = N_{G_j}(S_j)\setminus \{y\}$.
	\item $Z_3(g|_{X_j}) = Z_3(g) \setminus\{y\}$, then $Z_3(g|_{X_j}) \cap N_{G_j}[S_j] = \left( Z_3(g) \cap N_{G_i}[S_i] \right) \setminus\{y\} = \emptyset$.
	\item $Z_4(g|_{X_j}) = Z_4(g) \setminus\{y\}$, then $Z_4(g|_{X_j}) \cap N_{G_j}[S_j] = \left( Z_4(g) \cap N_{G_i}[S_i] \right) \setminus\{y\} = \emptyset$.
	\item $V_j - N_{G_j}[S_j] - Z_3(g|_{X_j}) \subseteq V_i - N_{G_i}[S_i] - Z_3(g)$ is an independent set.
\end{enumerate}
Then we have $D(g|_{X_j}) \le |S_j| = D(g)$.

Then we claim that if $D(g|_{X_j}) < +\infty$, let $g' \in \mathcal{G}(X_j)$ be a guessing such that $D(g') = D(g|_{X_j})$, then $D(g) \le D(g|_{X_j})$.
Let $S_j \subseteq V_j$ be a partial ve-dominating set of $g|_{X_j}$ on $G_j$ such that $D(g|_{X_j}) = |S_j|$.
Then we would like to prove that $S_i \triangleq S_j$ is a partial ve-dominating set of $g$ on $G_i$.
Let us verify the properties in Definition~\ref{def: partial ve-dominating set} one by one.
\begin{enumerate}
	\item $S_i \cap X_i = S_j \cap X_j = Z_1(g|_{X_j}) = Z_1(g)$.
	\item If $g(y) \neq 2$, then $Z_2(g) = Z_2(g|_{X_j}) \subseteq N_{G_j}(S_j) \subseteq N_{G_i}(S_i)$. If $g(y) = 2$, then $N(y) \cap Z_1(g) \neq \emptyset$, then $Z_2(g) = Z_2(g|_{X_j}) \cup \{y\} \subseteq N_{G_i}(S_i)$.
	\item If $g(y) \neq 3$, then $Z_3(g) = Z_3(g|_{X_j})$. Note that $N_{G_j}[S_j] \subseteq N_{G_i}[S_i]\cup \{y\}$. Then $Z_3(g) \cap N_{G_i}[S_i] = \left( Z_3(g|_{X_j}) \cap N_{G_j}[S_j] \right) = \emptyset$.
	If $g(y) = 3$, then $Z_3(g) = Z_3(g|_{X_j}) \cup \{y\}$. If $Z_3(g) \cap N_{G_i}[S_i] \neq \emptyset$, then the only possibility is $y \in N_{G_i}[S_i]$.
	By Corollary~\ref{coro: properties of N(y) for introduce node}, $y$ is connected to a vertex in $X_i \cap S_i$ which should be labelled $1$, a contradiction to the validity of $g$.
	\item The proof of the statement that $Z_4(g) \cap N_{G_i}[S_i] = \emptyset$ is the same as the last item by replacing $Z_3$ with $Z_4$.
	\item Similarly, the statement that $V_i - N_{G_i}[S_i] - Z_3(g)$ is an independent set only fails when $g(y) \neq 3$ and $y \notin N_{G_i}[S_i]$. In this case, we have $g(y) = 4$ and $y$ is connected to a vertex $z$ in $V_i - N_{G_i}[S_i] - Z_3(g)$. 
	By Corollary~\ref{coro: properties of N(y) for introduce node}, $z \in X_i$, then $g(z) = 4$, which contradicts the validity of $g$.
\end{enumerate}

In conclusion, we have proved (\ref{eq: introduce node}) in all cases.
Then we can calculate $D(g)$ for all valid guessings $g \in \mathcal{G}(X_i)$ in time $O(4^{k+1}k^2)$ using (\ref{eq: introduce node}).

\textbf{Join node:}
Suppose $i $ is a join node, and let $j$ and $k$ be its two child nodes. Let $g \in \mathcal{G}(X_i)$, $g_j \in \mathcal{G}(X_j)$ and $g_k \in \mathcal{G}(X_k)$.
We say $g_j$ and $g_k $ \textit{divide} $g$ if the following conditions hold.
\begin{enumerate}
	\item For any $x \in Z_1(g) \cup Z_3(g) \cup Z_4(g)$, $g_j(x) = g_k(x) = g(x)$.
	\item For any $x \in Z_2(g) \cap N(Z_1(g))$, $g_j(x) = g_k(x) = 2$.
	\item For any $x \in Z_2(g)$ but $x \notin N(Z_1(g))$, $\{g_j(x), g_k(x)\} = \{2,3\}$, that is, one of $g_j(x), g_k(x)$ is two and the other is three.
\end{enumerate}

We claim that for any valid $g \in \mathcal{G}(X_i)$, 
let
\[
	F(g) \triangleq \min \{D(g_j) + D(g_k) - |Z_1(g)| ~|~ g_j \in \mathcal{G}(X_j),g_k \in \mathcal{G}(X_k), g_j  \text{ and } g_k\text{ divide } g\},
\]
then we have
\begin{equation}\label{eq: D(g) join node}
	D(g) = F(g).
	% D(g) = \left\{\begin{aligned}
	% 	&\min \{D(g') + D(g'') - |Z_1(g)| ~|~ g' \in \mathcal{G}(X_j),g'' \in \mathcal{G}(X_k)\}  &\qquad& \text{ if } g'  \text{ and } g''\text{ divide } g, \\
	% 	&+\infty &\qquad& \text{ otherwise.}
	% \end{aligned}
	% \right.
\end{equation}

First, we claim that if $D(g) < +\infty$, then $F(g) \le D(g)$.
Let $S_i\subseteq V_i$ be a partial ve-dominating set of $g$ on $G_i$ such that $D(g)=|S_i|$. 
Let $S_j = S_i \cap V_j$ and $S_k = S_i \cap V_k$. 
Then $S_i = S_j \cup S_k$ and $|S_i| = |S_j| + |S_k| - |S_i \cap X_i|$ by Definition~\ref{def: tree-decomposition}.
Then we would like to prove that $S_j$ is a partial ve-dominating set of $g_j$ on $G_j$ for some $g_j$ and $S_k$ is a partial ve-dominating set of $g_k$ on $G_k$ for some $g_k$.
Define respectively $g_j \in \mathcal{G}(X_j)$ and $g_k \in \mathcal{G}(X_k)$ as following:
\[g_j(x) = \left\{\begin{aligned}
&g(x) &\qquad& \text{ if } g(x)\in Z_1(g)\cup Z_3(g)\cup Z_4(g), \\
&2 &\qquad& \text{ if } x\in Z_2(g) \text{  and } N(x) \cap S_j \neq \emptyset,\\
&3 &\qquad& \text{ if } x\in Z_2(g) \text{  and } N(x) \cap S_j = \emptyset,
\end{aligned}
\right.\]
\[g_k(x) = \left\{\begin{aligned}
&g(x) &\qquad& \text{ if } g(x)\in Z_1(g)\cup Z_3(g)\cup Z_4(g), \\
&3 &\qquad& \text{ if } x\in Z_2(g) \text{  and } N(x) \cap S_k \neq \emptyset,\\
&2 &\qquad& \text{ if } x\in Z_2(g) \text{  and } N(x) \cap S_k = \emptyset.
\end{aligned}
\right.\]
Then since $S_j \cap S_k = S_i \cap X_i = Z_1(g)$, it is easy to verify that $g_j$ and $g_k$ divide $g$.
We first show that $S_j$ is a partial ve-dominating set of $g_j$ on $G_j$ by verifying the properties of Definition~\ref{def: partial ve-dominating set}.
\begin{enumerate}
	\item $S_j \cap X_j = S_i \cap V_j \cap X_j = S_i \cap X_i = Z_1(g) = Z_1(g_j)$.
	\item $Z_2(g_j) \subseteq N_{G_j}(S_j)$ follows from the definition of $g_j$.
	\item $Z_3(g_j) \cap N[S_j] \subseteq Z_3(g) \cap N[S_i] = \emptyset$.
	\item $Z_4(g_j) \cap N[S_j] \subseteq Z_4(g) \cap N[S_i] = \emptyset$.
	% \item $Z_4(g_j) \subseteq \left(V_i - N_{G_i}[S_i] \right) \cap V_j \subseteq V_j - N_{G_j}[S_j]$.
	\item We claim that $V_j - N_{G_j}[S_j] - Z_3(g_j)$ is an independent set.
	Otherwise, assume there are two vertices $z_1,z_2 \in V_j - N_{G_j}[S_j] - Z_3(g_j)$ connected by an edge.
	If $z_1,z_2 \in X_i$, then $g(z_1) = g(z_2) = 4$, a contradiction to the validity of $g$.
	Then one of $\{z_1,z_2\}$ is in $V_j \setminus X_i$, say $z_1$.
	By Corollary~\ref{coro: properties of edge connectivity in join node}, $z_1$ cannot be connected to a vertex in $V_k\setminus X_i$, then $z_1 \in V_i - N_{G_i}[S_i] - Z_3(g)$.
	Since $V_i - N_{G_i}[S_i] - Z_3(g)$ is an independent set, $z_2$ must be outside the set, that is, $z_2 \in N_{G_i}[S_i] \cup Z_3(g)$.
	Recall that $z_2 \in V_j - N_{G_j}[S_j] - Z_3(g_j)$.
	If $z_2 \in S_i$, then $z_2 \in S_j$, a contradiction.
	If $z_2 \in Z_3(g)$, then $z_2 \in Z_3(g_j)$, a contradiction.
	If $z_2 \in N(S_i)\setminus S_i$ and $z_2 \in X_i$, then $g(z_2) = 2$ by definition of the labeling on $X_i$, so $z_2 \in N_{G_j}[S_j] \cup Z_3(g_j)$, a contradiction.
	Finally, if $z_2 \in N(S_i)\setminus S_i$ and $z_2 \notin X_i$, then $z_2 \in N_{G_j}[S_j]$, a contradiction.
	In all cases, we have a contradiction.
\end{enumerate}
Thus $S_j$ is a partial ve-dominating set of $g_j$ on $G_j$.
Similarly, we can show $S_k$ is a partial ve-dominating set of $g_k$ on $G_k$.
So $F(g) \le |S_j| + |S_k| - |Z_1(g)| = D(g)$ by the definition of $F(g)$.

Then we claim that if $F(g) < +\infty$, then $D(g) \le F(g)$.
Let $g_j \in \mathcal{G}(X_j)$ and $g_k \in \mathcal{G}(X_k)$ such that $g_j$ and $g_k$ divide $g$ and $F(g) = D(g_j) + D(g_k) - |Z_1(g)|$.
Let $S_j \subseteq V_j$ and $S_k \subseteq V_k$ be partial ve-dominating sets of $g_j$ on $G_j$ and $g_k$ on $G_k$ respectively such that $|S_j| = D(g_j)$ and $|S_k| = D(g_k)$.
Then we would like to prove that $S_i \triangleq S_j \cup S_k$ is a partial ve-dominating set of $g$ on $G_i$.
Note that the size of $S_i$ is $|S_i| = |S_j| + |S_k| - |S_j \cap S_k| = D(g_j) + D(g_k) - |Z_1(g)| = F(g)$.
Again, let us verify the properties of Definition~\ref{def: partial ve-dominating set} one by one.

\begin{enumerate}
	\item $S_i \cap X_i = (S_j\cap X_i) \cup (S_k\cap X_i) = Z_1(g_j) \cup Z_1(g_k) = Z_1(g)$ since $g_j$ and $g_k$ divide $g$.
	\item $Z_2(g) = Z_2(g_j) \cup Z_2(g_k) \subseteq N_{G_j}(S_j) \cup N_{G_k}(S_k) \subseteq N_{G_i}(S_i)$.
	\item $Z_3(g) = Z_3(g_j) \cap Z_3(g_k)$ and $N[S_i] = N[S_j] \cup N[S_k]$. Then $Z_3(g) \cap N[S_i] = \emptyset$.
	\item $Z_4(g) = Z_4(g_j) = Z_4(g_k)$ and $N[S_i] = N[S_j] \cup N[S_k]$. For every vertex $x \in Z_4(g)$, $x \notin N_{G_j}[S_j] \cup N_{G_k}[S_k] =N_{G_i}[S_i]$, then $Z_4(g) \cap N_{G_i}[S_i] = \emptyset$.
	\item We claim that $V_i - N_{G_i}[S_i] - Z_3(g)$ is an independent set.
	Suppose otherwise, let $z_1, z_2$ be two vertices in $V_i - N_{G_i}[S_i] - Z_3(g)$ connected by an edge.
	If $z_1, z_2 \in X_i$, then $g(z_1) = g(z_2) = 4$, a contradiction to the validity of $g$.
	Then one of $z_1$ is in $V_i \setminus X_i$, say $z_1$, without loss of generality, assume $z_1 \in V_j$.
	Then $z_1 \in V_j - N_{G_j}[S_j] - Z_3(g_j)$.
	By Corollary~\ref{coro: properties of edge connectivity in join node}, $z_2 \in V_j$.
	Since $V_j - N_{G_j}[S_j] - Z_3(g_j)$ is an independent set, we have $z_2 \notin N_{G_j}[S_j] \cup Z_3(g_j)$.
	Then $z_2 \in X_i$, $g_j(z_2) = 3$ and $g(z_2) \neq 3$.
	It only happens when $g(z_2) = 2$ since $g_j$ and $g_k$ divide $g$.
	Then $z_2 \in N[S_i]$, a contradiction.
\end{enumerate}
Hence $S_i $ is a partial ve-dominating set of $g$ on $G_i$ and then we have $D(g) \le |S_i| = |S_j| + |S_k| - |Z_1(g)| = F(g)$.
Thus (\ref{eq: D(g) join node}) holds.

	If $|Z_2(g)| = m$, there are at most $2^m$ pairs of $g_j$ and $g_k$ that divide $g$.
	There are $\binom{k+1}{m} 3^{k+1-m}$ guessings $g$ such that $|Z_2(g)|=m$. So the evaluation of $D(g)$ for a join node can be carried out in time
	 $O\left( \sum \limits_{0 \le m \le k+1}2^m \binom{k+1}{m} 3^{k+1-m} \right) =  O(5^{k+1}) $.
	
% finish proof

\noindent{\bf Step 3. } Let $r$ denote the root of $T$.
We finally output:
\begin{equation}
	\gamma_{ve}(G)=\min \{D(g) \mid \mbox{$g \in \mathcal{G}(X_r)$, $g$ is valid and $Z_3(g) = \emptyset$} \}.
\end{equation}

From above discussion, the total running time of the algorithm is $O(n5^{k+1})$.
It is worth mentioning that the above algorithm only calculates the ve-domination number.
If we want one or all ve-dominating sets with minimum cardinality, we only need to store the guessing relationship at each node.
We can also calculate one or all ve-dominating sets with minimum cardinality from root to leaves after computing $D(g)$ at each node.
From all of the above, Theorem~\ref{thm: main thm: poly algorithm of ve-dominating set} is proved.
$\square$

%In addition, we can modify the count method to solve the weighted version minimum ve-dominating set.
%It will bring no additional computing cost.

% ------------------------ Planar Graph Treewidth ----------------------
\section{Treewidth and ve-domination on planar graphs}\label{sec: planar treewidth ve-domination number}

In this section, we consider the relationship between treewidth and ve-domination number on planar graphs. 
The methods are generalized from~\cite{ABo}.
We first give some definitions.
\vskip.2cm
\begin{definition}
	A crossing-free embedding of a graph $G$ in the plane is called \textbf{outerplanar} if each vertex lies on the boundary of the outer face.
	A graph $G$ is called \textbf{outerplanar} if it admits an outerplanar embedding in the plane.
\end{definition}

The following generalization of the notion of outer planarity can be found in~\cite{Baker}.

%\vspace{.2cm}

	A crossing-free embedding of a graph $G$ in the plane is called \textbf{$r$-outerplanar} if, for $r=1$, the embedding is outerplanar, and for $r > 1$, inductively, when removing all vertices on the boundary of the outer face and their incident edges, the embedding of the remaining subgraph is $(r-1)$-outerplanar.
	A graph $G$ is called \textbf{$r$-outerplanar} if it admits an $r$-outerplanar embedding.
	The smallest number $r$, such that $G$ is $r$-outerplanar, is called the \textbf{outerplanarity} number. For a given $r$-outerplanar embedding of a graph $G(V,E)$, we define the $i$-th layer $L_i$ inductively as follows.
	Layer $L_1$ consists of the vertices on the boundary of the outer face, and for $i > 1$, layer $L_i$ is the set of vertices that lie on the boundary of the outer face in the embedding of the subgraph $G-(L_1 \cup \cdots \cup L_{i-1})$.

The relationship between $r$-outerplanarity and treewidth is shown in Theorem~\ref{thm: outerplanarity treewidth}.

%\vspace{.2cm}
\begin{theorem}[Theorem 83 of \cite{PlanarityTreewidth}]\label{thm: outerplanarity treewidth}
	An $r$-outerplanar graph has treewidth of at most $3r-1$.
\end{theorem}

%\vspace{.2cm}

A maximal connected subgraph without a cut vertex of $G$ is called a \textit{block}.
A block is either a maximal biconnected subgraph with at least $3$ vertices, or an edge, or a single vertex.
A {\em layer decomposition} of an $r$-outerplanar embedding of graph $G$ is a forest of height $r-1$.
The nodes of the forest correspond to different blocks of the subgraphs of $G$ induced by a layer.
For each layer with vertex set $L_i$, suppose the blocks of the subgraph of $G$ induced by $L_i$ have vertex sets $C_{i,1}, \ldots, C_{i, l_i}$. Each $C_{i,j}$ is called a \textit{layer component}.
Then $L_i = \bigcup_{j=1}^{l_i} C_{i,j}$ and we have $l_i$ nodes that represent the nodes of layer $L_i$, one for each such block.

% \begin{figure}[t]
% 	\centering         %使图片居中放置
% 	\includegraphics[width=0.7\linewidth]{./src/split.png}
% 	\caption{$C_{i,j}$ is split at cut vertex $x$ into four parts.}
% 	\label{fig: split}
% \end{figure}

% If $C_{i,j}$ has a cut vertex $x$ and let $A_1, A_2, \ldots, A_t$ be the components of $C_{i,j} - {x}$.
% Split $C_{i,j}$ into $C_{i,j} = G[A_k \cup \{x\}]$ (see Figure~\ref{fig: split}).
% If $C_{i,j}$ is not biconnected, split in the same way until all parts are biconnected.
% All split biconnected parts are denoted by $\{C_{i,j} \mid 1 \le k \le t_{i,j}\}$ where $t_{i,j}$ is the number of split biconnected parts split from $C_{i,j}$.

The layer components have the following properties due to the definition of blocks:

\begin{enumerate}
	\item If $|C_{i,j}| > 2$ , then $C_{i,j}$ is biconnected.
	\item $|C_{i,j} \cap C_{i,j'}| \le 1$ for any $j \neq j'$.
	% \item There are at most one vertices in $C_{i,j} \cap C_{i,j'}$ for any $j \neq j'$.
\end{enumerate}

Each layer component node $C_{1,j}$ is the root of a tree in the forest.
%  so we will have one tree per block of $G[L_i]$, representing the exterior vertices of the block.
Two layer component nodes will be adjacent if one is connected to the other from inside.
Here we say $C_{i',j'}$ is connected to $C_{i,j}$ from inside if:
\begin{enumerate}
	\item $i' = i + 1$, and $C_{i,j}$ is neither a single vertex nor an edge.
	\item no vertices in $C_{i',j'}$ lie on the boundary of the outer face in $G'= G[(V-\bigcup_{t=1}^{i}L_t) \cup C_{i,j}]$,
	\item and there exists an edge between $C_{i',j'}$ and $C_{i,j}$.
\end{enumerate}

This means that a layer component node $C_{i,j}$ can be adjacent only to layer component nodes of the form $C_{i-1,j'}$ or $C_{i+1,j''}$.
If $C_{i,j}$ is adjacent to $C_{i-1,j'}$, then the vertices of $C_{i,j}$ lie within the area formed by the subgraph induced by $C_{i-1,j'}$.
Note that the layer component nodes on the $i$-th level of the forest correspond to the layer components of the form $C_{i,j}$.
One easily observes that the planarity of $G$ implies that the layer decomposition must indeed be a forest.

We need some further notation.
	A layer component $C_{i,j}$ of layer $L_i$ is called \textbf{non-vacuous} if there are vertices from layer $L_{i+1}$ in the interior of $C_{i,j}$ (i.e., in the region enclosed by the subgraph induced by $C_{i,j}$).
	So $C_{i,j}$ is non-vacuous iff the corresponding component node in the layer decomposition has a child.
	A layer component $C_{i,j}$ of layer $L_i$ is called \textbf{$s$-non-vacuous} if,
	for $s=1$, $C_{i,j}$ is non-vacuous, and for $s > 1$, inductively, there exists a $(s-1)$-non-vacuous layer component from layer $L_{i+1}$ in the interior of $C_{i,j}$ (i.e., in the region enclosed by the subgraph induced by $C_{i,j}$).
	So $C_{i,j}$ is $s$-non-vacuous iff the corresponding component node in the layer decomposition has distance $s-1$ to a leaf node. 
	We have the following result.

%\vspace{.2cm}
\begin{lemma}\label{lemma: boundary cycle}
	Let $\emptyset \neq C \subseteq C_{i,j}$ be a subset of a non-vacuous layer component $C_{i,j}$ of layer $i$, where $i \ge 2$.
	Then there exists a unique smallest (in number of vertices) cycle $B(C)$ (which is called the boundary cycle of $C$) in layer $L_{i-1}$, such that $C$ is contained in the region enclosed by $B(C)$.
	No other vertex of layer $L_{i-1}$ is contained in this region. And there exists a layer component $C_{i-1,j'}$ that contains $B(C)$.
\end{lemma}

\begin{proof}
	From Lemma~6 of \cite{ABo}, there exists such cycle $B(C)$ in layer $L_{i-1}$.
	Since $B(C)$ is biconnected, it must belong to a block of $G[L_{i-1}]$, that is, a layer component $C_{i-1,j'}$.
\end{proof}

%\vspace{.2cm}
%\begin{definition}
	%For each non-empty subset $C$ of a non-vacuous layer component of layer $i$ ($i \ge 2$), the set $B(C)$ as given in Lemma~\ref{lemma: boundary cycle} is called the boundary cycle of $C$.
%\end{definition}

% We assume that we have a ve-dominating set $D$ of size at most $k$.
Assume $D$ is a ve-dominating set of an $r$-outerplanar graph $G$ with size at most $k$.
Let $k_i$ be the number of vertices of $D_i = D \cap L_i$.
Hence, $\sum_{i=1}^{r} k_i \le k$.
% In addition, set $k_0 = k_{r+1} = k_{r+2} = 0$.
Moreover, let $c_i$ denote the number of $3$-non-vacuous layer components of layer $L_i$. Let $G$ be a graph. A subset $S\subseteq V(G)$ is called a
separator of $G$, if the subgraph $G - S$ is disconnected.

%\vspace{.2cm}
\begin{prop}[\cite{ABo}]\label{prop: treewidth layer decomposition}
	Let $G$ be a plane graph with layers $L_i~(i = 1,\ldots, r)$.
	For $i=1,\ldots,l$, let $\mathcal{L}_i$ be a set of consecutive layers, i.e., $\mathcal{L}_{i} = \{L_{j_i}, L_{j_i+1}, \ldots L_{j_i+n_i}\}$, such that $\mathcal{L}_i \cap \mathcal{L}_{i'} = \emptyset$ for all $i \neq i'$.
	Moreover, suppose $G$ can be decomposed into components, each of treewidth of at most $t$, by means of separators $S_1,\ldots,S_l$, where $S_i \subseteq \bigcup_{L \in \mathcal{L}_i} L$ for all $i = 1,\ldots, l$.
	Then $G$ has a treewidth of at most $t + 2s$, where $s = \max \limits_{1 \le i \le l} |S_i|$.
\end{prop}

\begin{figure}[t]
	\centering         %使图片居中放置
	\includegraphics[width=0.5\linewidth]{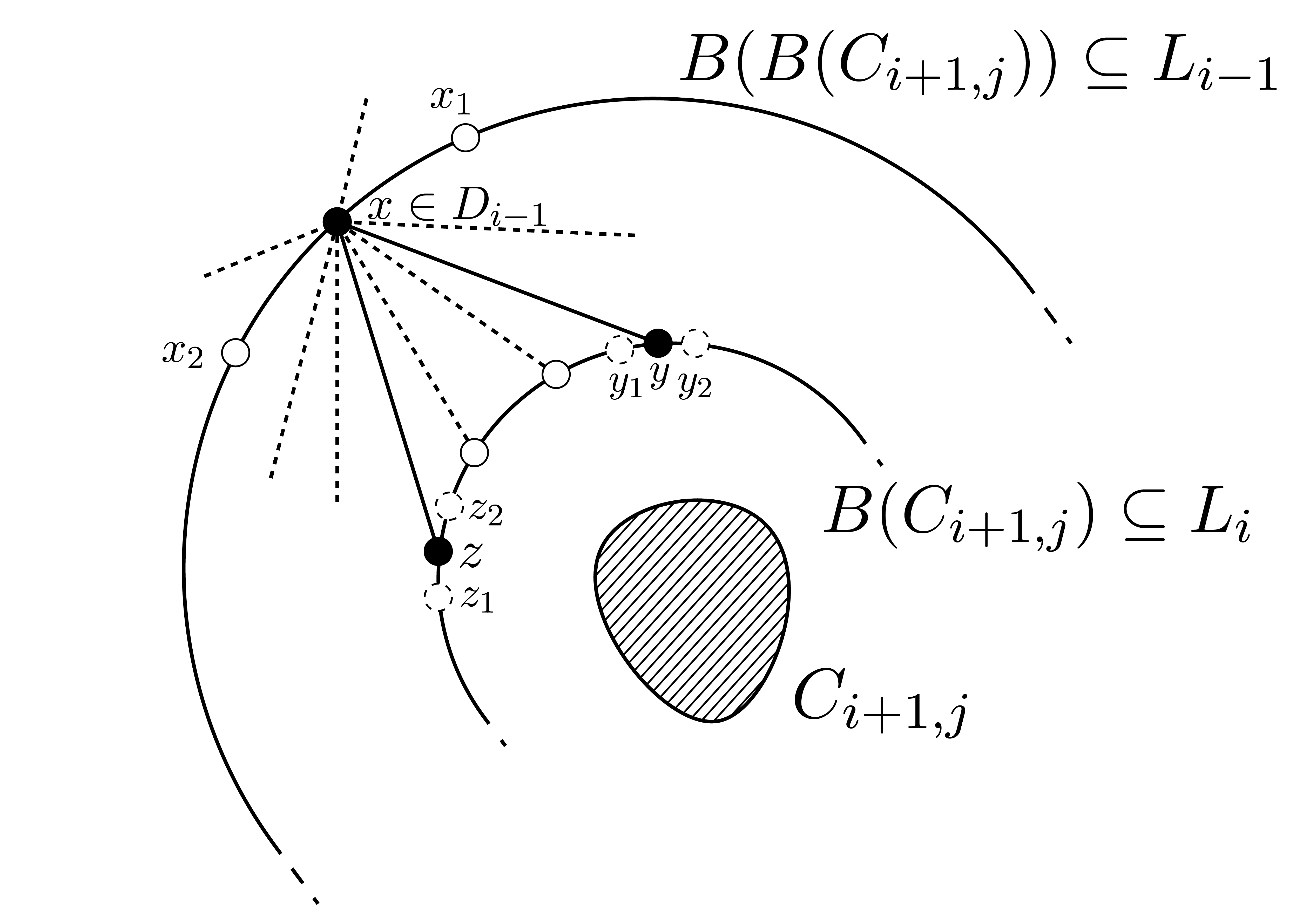}
	\caption{(Generalized) Upper Triples.}
	\label{fig: upper triples}
\end{figure}

Proposition~\ref{prop: treewidth layer decomposition} is a useful tool when estimating the treewidth of $G$.
%Then we are going to find separators layerwisely.
In \cite{ABo}, \textit{Upper Triples, Middle Triples, Lower Triples} were defined on non-vacuous layer components. Here we extend these three concepts. All the new concepts including upper triple, middle triple and lower triples are defined on $3$-non-vacuous layer components.

%\vspace{.2cm}
\noindent
\textbf{(Generalized) Upper Triples:} A (generalized) upper triple for layer $L_i$ is associated to a $3$-non-vacuous layer component $C_{i+1,j}$ of layer $L_{i+1}$ and a vertex $x \in D_{i-1}$ that has a neighbor on the boundary cycle $B(C_{i+1,j}) \subseteq L_i$~(see Figure~\ref{fig: upper triples}). By the definition of a boundary cycle, $x \in B(B(C_{i+1,j}))$. Let $x_1,x_2\in N(x)\cap B(B(C_{i+1,j}))$.
Starting from $x_1$, we go around $x$ up to $x_2$ so that we visit all neighbors of $x$ in layer $L_i$.
% We note the neighbors of $x$ on the boundary cycle $B(C_{i+1,j})$.
Going around gives two outermost neighbors $y$ and $z$ on this boundary cycle.
 If $x$ has only a single neighbor $y$ in $B(C_{i+1,j})$, let $z=y$.
%Let the \textit{generalized upper triple} be $\{x,y,z\} \cup ( (N(y) \cup N(z)) \cap B(C_{i+1,j}))$. 
We call the set $\{x,y,z\}$ (resp.\ $\{x,y,z\} \cup ( (N(y) \cup N(z)) \cap B(C_{i+1,j}))$) an \textit{upper triple (resp.\ a generalized upper triple)} of layer $L_i$.%That is, we add the neighbors of $y$ and $z$ on $B(C_{i+1,j})$ to the generalized upper triple.

%\vspace{.2cm}
%\begin{definition}
	%For each $3$-non-vacuous layer component $C_{i+1,j}$ of $L_{i+1}$ and
	%each vertex $x \in D_{i-1}$ with neighbors in $B(C_{i+1,j})$,
	%the set $\{x,y,z\}$ (resp. $\{x,y,z\} \cup ( (N(y) \cup N(z)) \cap B(C_{i+1,j}))$) as described above is called an \textit{upper triple (resp. generalized upper triple)} of layer $L_i$.
%\end{definition}

\begin{figure}[t]
	\centering         %使图片居中放置
	\begin{minipage}{0.45\linewidth}
		\centering
		\includegraphics[width=\linewidth]{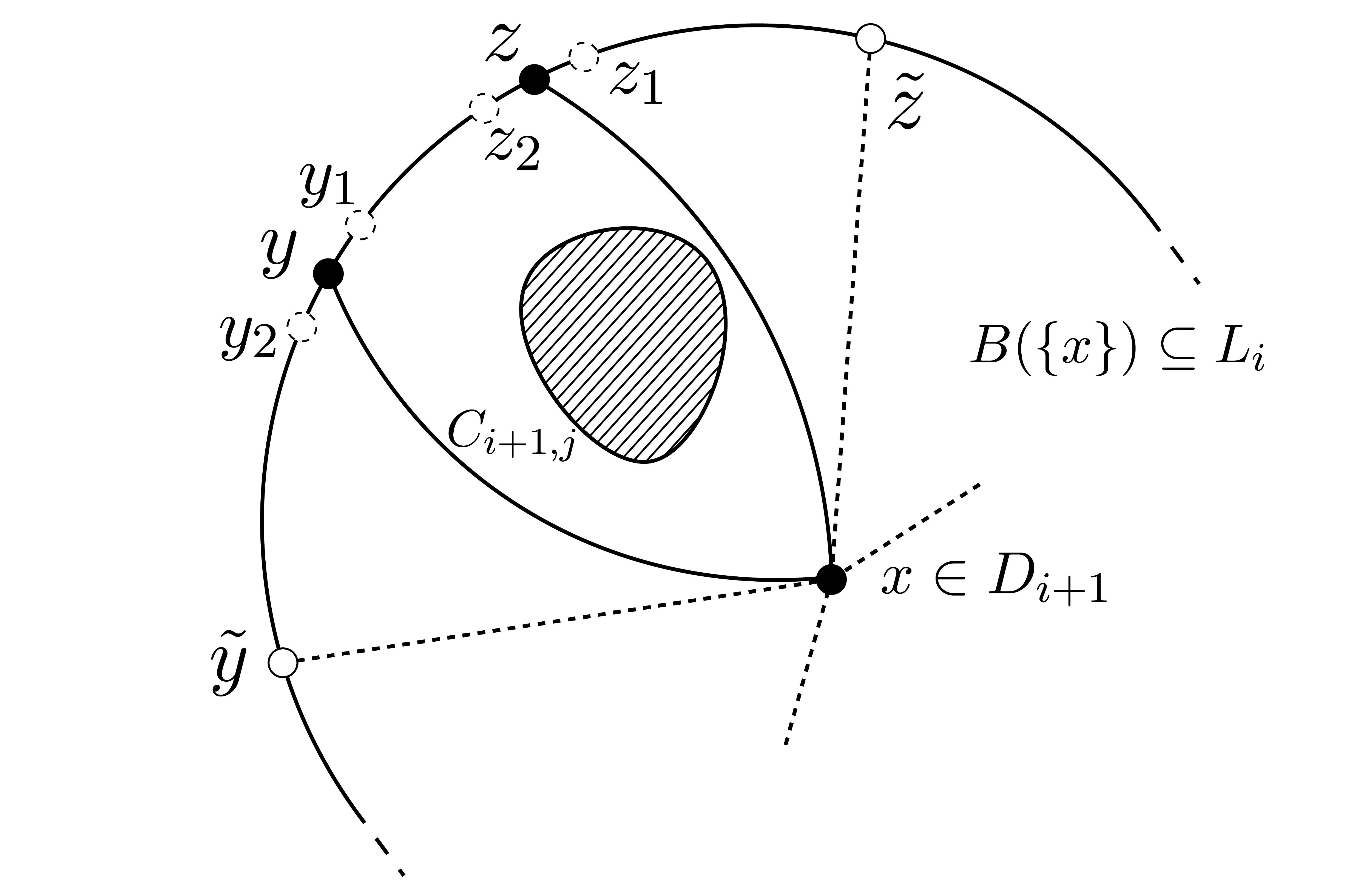}
		\caption{(Generalized) Lower Triples.}
		\label{fig: lower triples}
	\end{minipage}
	\begin{minipage}{0.45\linewidth}
		\centering
		\includegraphics[width=\linewidth]{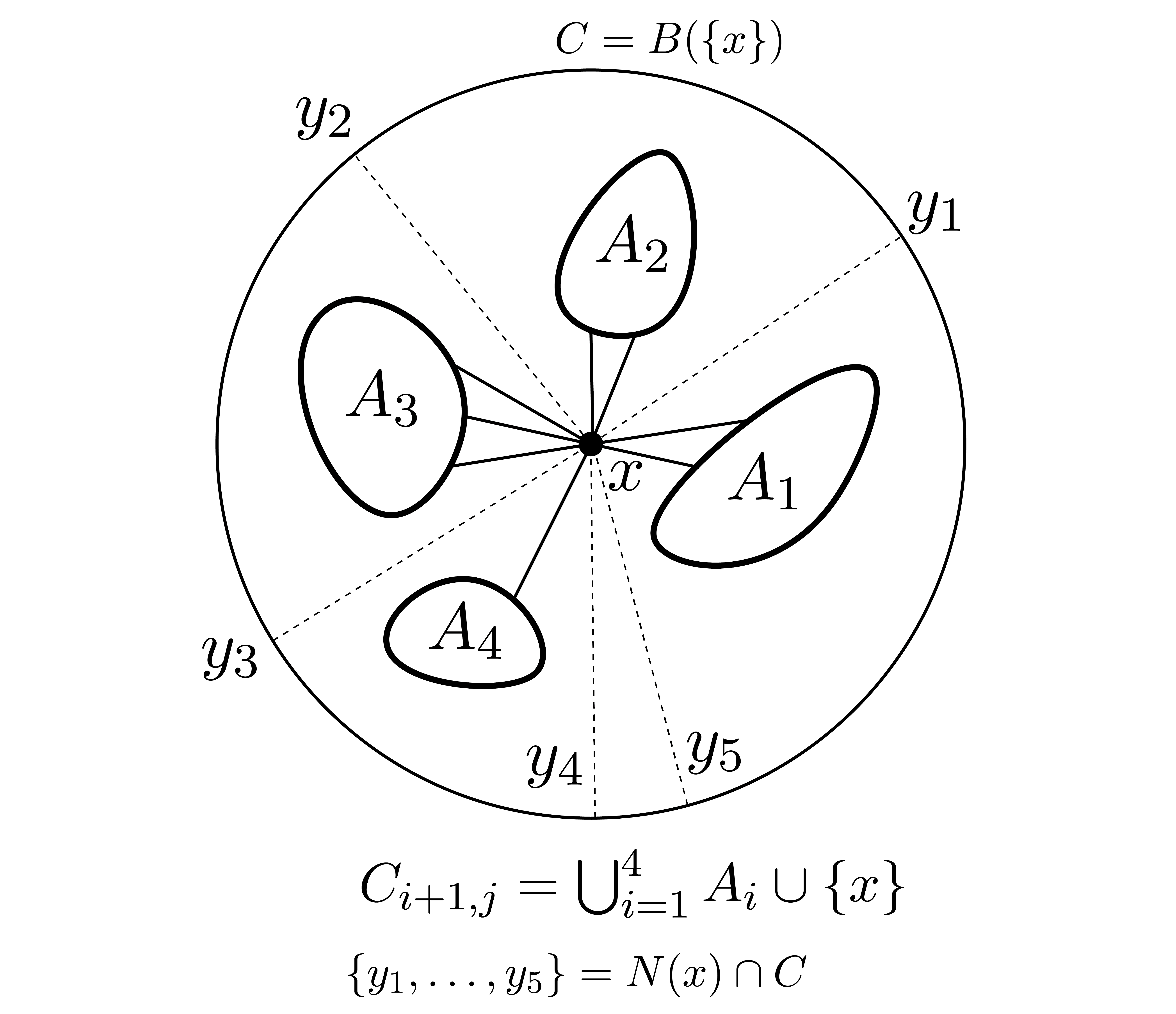}
		\caption{An example when $\tilde{y}, \tilde{z}$ do not exist.}
		\label{fig: lower triples contradiction}
	\end{minipage}
\end{figure}

% \begin{figure}[t]
% 	\centering         %使图片居中放置
% 	\includegraphics[width=0.6\linewidth]{./src/lower_contradiction.png}
% 	\caption{(Generalized) Lower Triples.}
% 	\label{fig: lower triples contradiction}
% \end{figure}

%\vspace{.2cm}
\noindent
\textbf{(Generalized) Lower Triples:} A lower triple for layer $L_i$ is associated with a vertex $x \in D_{i+1}$ and a $3$-non-vacuous layer component $C_{i+1,j}$ of layer $L_{i+1}$.
We only consider layer components $C_{i+1,j}$ of layer $L_{i+1}$ that are enclosed by the boundary cycle $B(\{x\})$.
For each pair $\tilde{y}, \tilde{z} \in B(\{x\}) \cap N(x)$ (where $\tilde{y} \neq \tilde{z}$), we consider the path $P_{\tilde{y},\tilde{z}}$ from $\tilde{y}$ to $\tilde{z}$ along the cycle $B(\{x\})$, taking the direction such that the region enclosed by $\{\tilde{z}, x\}$, $\{x, \tilde{y}\}$ and $P_{\tilde{y}, \tilde{z}}$ contains the layer component $C_{i+1,j}$ (see Figure~\ref{fig: lower triples}).

If $|N(x) \cap B(\{x\})| \ge 2$, we claim such $\tilde{y}, \tilde{z}$ must exist.
Otherwise, $\{xy \mid y \in N(x) \cap B(\{x\}) \}$ divides the interior of $B(\{x\})$ into several parts and $C_{i+1, j}$ lies in at least two of the parts.
Then $x$ must be a cut vertex of $C_{i+1,j}$, a contradiction.
An example is shown in Figure~\ref{fig: lower triples contradiction}\footnote{Compared with \cite{ABo}, our layer component is defined to be blocks rather than connected components of each layer. The purpose is to avoid this situation.}.
Let $\{y,z\} \subseteq B(\{x\}) \cap N(x)$ be the pair such that the corresponding path $P_{y,z}$ is shortest.
%The triple, then, is the three-element set $\{x,y,z\}$. 
If $x$ has no or only a single neighbor $y$ on $B(\{x\})$, then let $x=y=z$, or $y=z$ respectively.
%Let the \textit{generalized lower triple} be $\{x,y,z\} \cup ( (N(y) \cup N(z)) \cap B(\{x\}))$. 
We call the set $\{x,y,z\}$ (resp.\ $\{x,y,z\} \cup ( (N(y) \cup N(z)) \cap B(\{x\}))$) a \textit{lower triple (resp.\ generalized lower triple)} of layer $L_i$.%That is, we add the neighbors of $y$ and $z$ on $B(\{x\})$ to the generalized lower triple.

%\vspace{.2cm}
%\begin{definition}
	%For each vertex $x$ of $D_{i+1}$ and each $3$-non-vacuous layer component $C_{i+1,j}$ that is enclosed by $B(\{x\})$,
	%the set $\{x,y,z\}$ (resp. $\{x,y,z\} \cup ( (N(y) \cup N(z)) \cap B(\{x\})$)) as described above is called a \textit{lower triple (resp. generalized lower triple)} of layer $L_i$.	
%\end{definition}

\begin{figure}[t]
	\centering         %使图片居中放置
	\includegraphics[width=\linewidth]{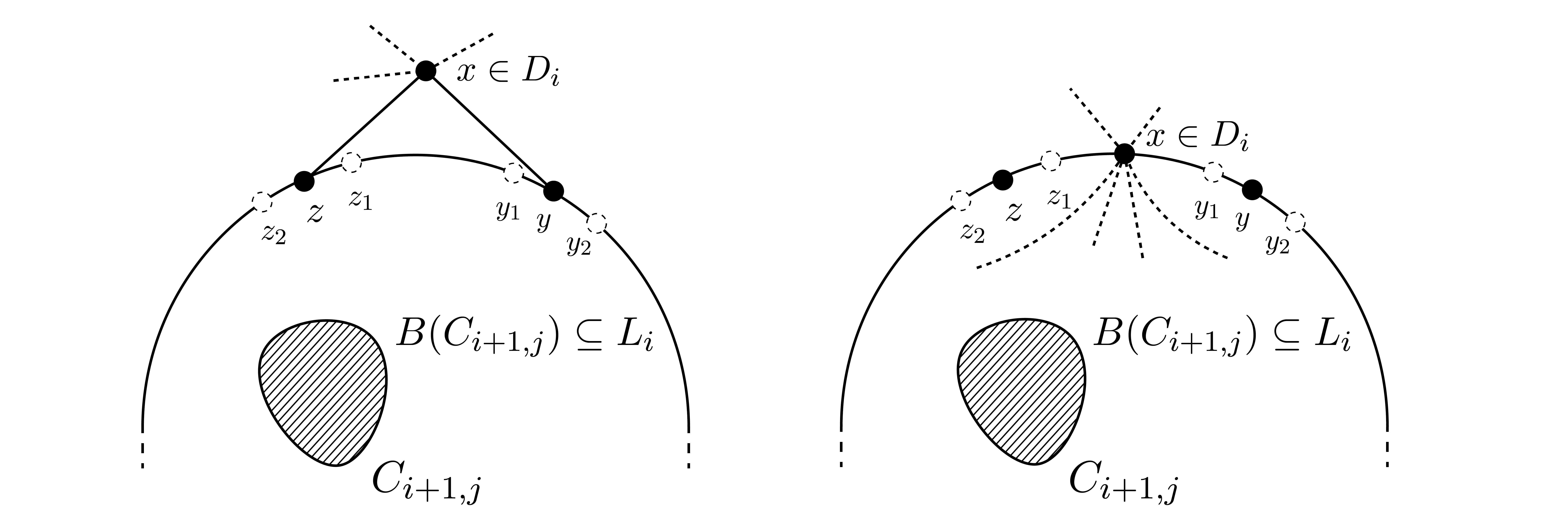}
	\caption{(Generalized) Middle Triples.}
	\label{fig: middle triples}
\end{figure}

%\vspace{.2cm}
\noindent
\textbf{(Generalized) Middle Triples:} A middle triple for layer $L_i$ is associated to a $3$-non-vacuous layer component $C_{i+1,j}$ and a vertex $x \in D_i$ that has a neighbor in $B(C_{i+1,j})$ (see Figure~\ref{fig: middle triples}).
Note that, due to the layer model, it is easy to see that a vertex $x \in D_i$ can have at most two neighbors $y,z$ in $B(C_{i+1,j})$.
Depending on whether $x$ itself lies on the cycle $B(C_{i+1,j})$ or not, we obtain two different cases which are both illustrated in Figure~\ref{fig: middle triples}.
In either of these cases, the middle triple is defined as the set $\{x,y,z\}$ where $y$ and $z$ are neighbors of $x$ on $B(C_{i+1,j})$.
Again, if $x$ has none or only a single neighbor $y$ in $B(C_{i+1,j})$, then let $x=y=z$ or $y=z$, respectively.
%Let the \textit{generalized middle triple} be $\{x,y,z\} \cup ( (N(y) \cup N(z)) \cap B(C_{i+1,j}))$. 
We call the set $\{x,y,z\}$ (resp.\ $\{x,y,z\} \cup ( (N(y) \cup N(z)) \cap B(C_{i+1,j}))$)  a \textit{middle triple (resp.\ generalized middle triple)} of layer $L_i$.%That is, we add the neighbors of $y$ and $z$ on $B(C_{i+1,j})$ to the generalized middle triple.

%\vspace{.2cm}
%\begin{definition}
	%For each $3$-non-vacuous layer component $C_{i+1,j}$ and each vertex $x \in D_i$,
	%the set $\{x,y,z\}$ (resp. $\{x,y,z\} \cup ( (N(y) \cup N(z)) \cap B(C_{i+1,j})$)) as described above is called a \textit{middle triple (resp. generalized middle triple)} of layer $L_i$.	
%\end{definition}

%\vspace{.2cm}

	Let $S_i$ (resp.\ $S_i'$) be the union of all upper triples, lower triples, and middle triples (resp.\ generalized upper triples, generalized lower triples, and generalized middle triples) of layer $L_i$. Then we have the following result.

% In \cite{ABo}, it is shown when $D$ is a dominating set, $S_i$ separates vertices of layers $L_{i-1}$ and $L_{i+2}$.
% We now give a similar result when $D$ is a ve-dominating set as Proposition~\ref{prop: ve domination separates}.

\begin{figure}[t]
	\centering         %使图片居中放置
	\begin{minipage}{0.45\linewidth}
		\centering
		\includegraphics[width=\linewidth]{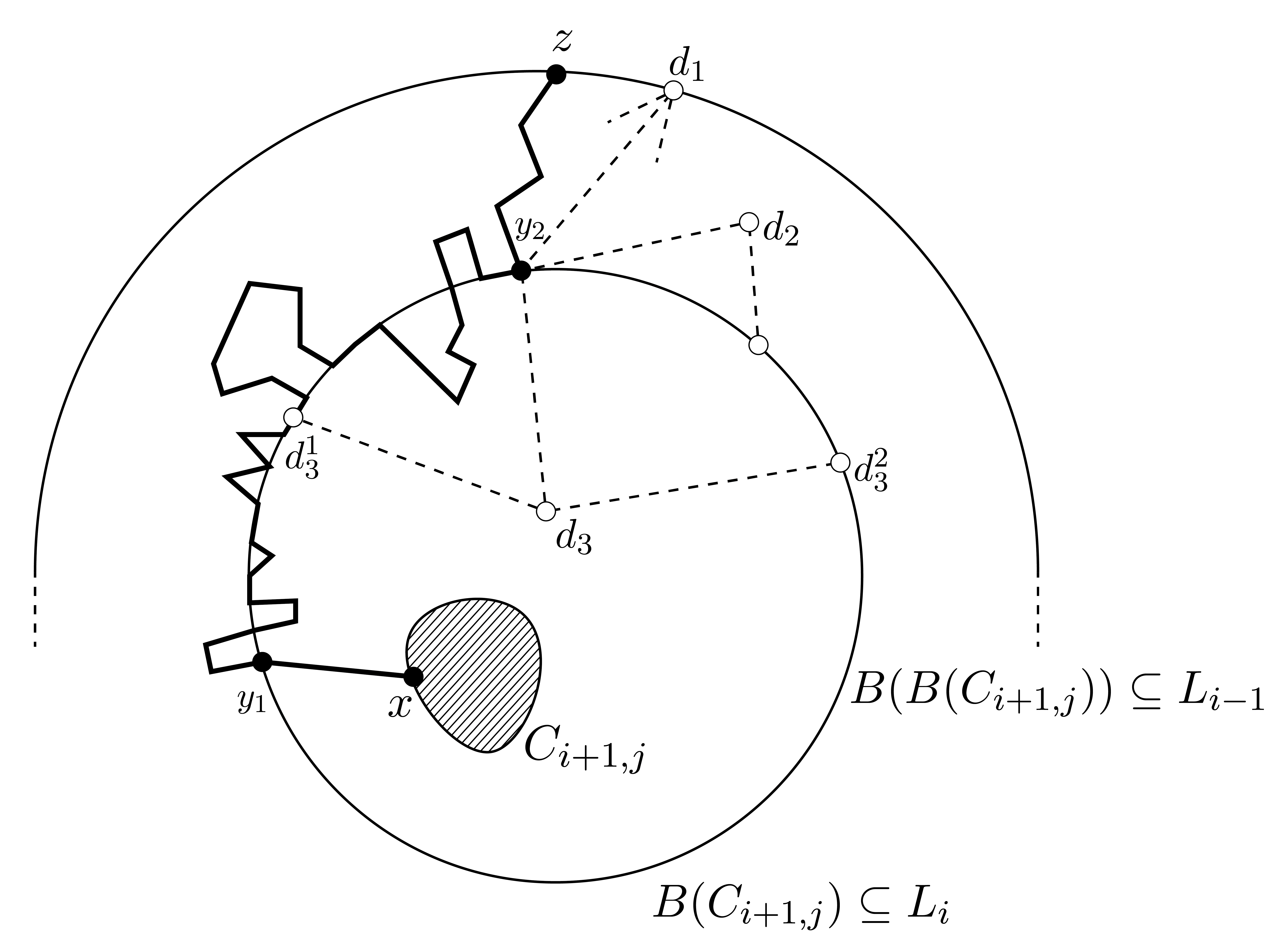}
		\caption{$S_i'$ separates $L_{i-1}$ and $L_{i+4}$ when $y_2$ is dominated.}
		\label{fig: separator 1}
	\end{minipage}
	\hspace{1em}
	\begin{minipage}{0.45\linewidth}
		\centering
		\includegraphics[width=\linewidth]{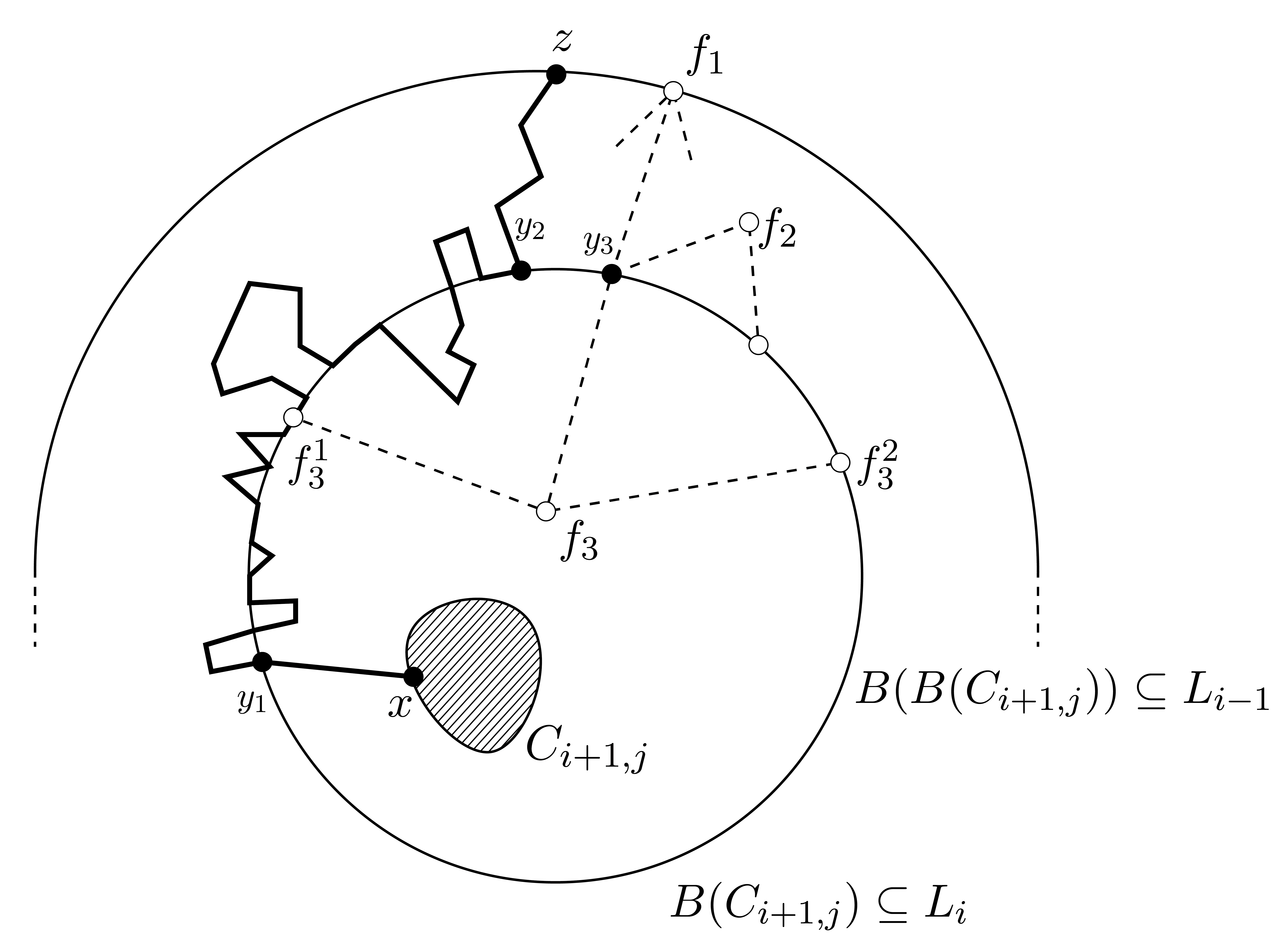}
		\caption{$S_i'$ separates $L_{i-1}$ and $L_{i+4}$ when $y_2$ is abandoned.}
		\label{fig: separator 2}
	\end{minipage}
\end{figure}

%\vspace{.2cm}
\begin{prop}\label{prop: ve domination separates}
	 $S_i'$ separates vertices of layers $L_{i-1}$ and $L_{i+4}$.
\end{prop}

\begin{proof}
	Suppose there is a path $P$ from layer $L_{i+4}$ to layer $L_{i-1}$ that avoids $S_i'$. Let $x$ be the last vertex from some $3$-non-vacuous layer component $C_{i+1,j}$  and $z$ be the first vertex from $B(B(C_{i+1,j}))$ in $P$.
	Then there exists a path $P'$ from $x$  to  $z $  which has the following properties:
	\begin{itemize}
		\item $P'$ avoids $S_i'$.
		\item All vertices in between $x$ and $z$ along $P'$ belong to layer $L_i$ or to vacuous layer components of layer $L_{i+1}$ or to $t$-non-vacuous layer components with $t <3$ of layer $L_{i+1}$.
	\end{itemize}
	Let $y_1$ (resp.\ $y_2$) be the first (resp.\ last) vertex along the path $P'$ from $x$ to $z$ that lies on the boundary cycle $B(C_{i+1,j}) \subseteq L_i$.
Since $	D\subseteq S_i'$, $y_2\notin D$.
	
	If $y_2$ is dominated by $D$ (see Figure~\ref{fig: separator 1}), from the proof of Proposition 25 of \cite{ABo}, we have $y_2\in S_i\subseteq S_i'$, a contradiction
	\footnote{ In Proposition 25 of \cite{ABo}, the triples are associated to non-vacuous layers. But the proof is essentially the same.}.

If $y_2$ is abandoned by $D$, let $y_3$ be a neighbor of $y_2$ on $B(C_{i+1, j})$, then $y_3$ must be dominated by some vertex in $D$.
	%We consider the vertex in $D$ that dominates $y_3$.
	This vertex can lie in layer $L_{i-1}$, layer $L_i$ or layer $L_{i+1}$.

	First suppose $y_3$ is dominated by a vertex $f_1 \in L_{i-1}$ (see Figure~\ref{fig: separator 2}). Then $f_1\in B(B(C_{i+1,j}))$ and $y_2f_1 \notin E$.
	Thus $y_3$ must be an ``outermost" neighbor of $f_1$ among all vertices in $N(f_1) \cap B(C_{i+1,j})$;
	otherwise there would be an edge from $f_1$ to a vertex on $B(C_{i+1,j})$ that leaves the closed region bounded by $\{f_1, y_3\}$, $\{y_3,y_2\}$, the path from $y_2$ to $z$, and the corresponding path from $z$ to $f_1$ along $B(B(C_{i+1,j}))$, a contradiction with  $G$ being planar.
	Hence, $y_3$ would be in the upper triple of layer $L_i$ which is associated to the layer component $C_{i+1,j}$ and $f_1$.
	Then $y_2$ must be in the generalized upper triple which contradicts the assumption that $P'$ avoids $S_i'$.

	Now, suppose $y_3$ is dominated by a vertex $f_2 \in D_i$ (see Figure~\ref{fig: separator 2}).
	By the definition of middle triple, this  implies that $y_3$ is in the middle triple associated to $C_{i+1,j}$ and $f_2$.
	Then $y_2$ must be in the corresponding generalized middle triple, a contradiction.

	Consequently, $y_3$ must be dominated by some vertex $f_3$ in layer $L_{i+1}$.
	Let $\{f_3, f_3^{1}, f_3^{2}\}$, where $f_3^1, f_3^2 \in N(f_3) \cap B(C_{i+1,j})$, be the lower triple associated to $C_{i+1,j}$ and $f_3$ (see Figure~\ref{fig: separator 2}).
	Then $y_2,y_3$ are not contained in the lower triple by $P'$ avoiding $S_i'$.
	By definition, $C_{i+1,j}$ is contained in the region enclosed by $\{f_3^1,f_3\}$, $\{f_3,f_3^2\}$ and the path from $f_3^2$ to $f_3^1$ along $B(C_{i+1,j})$.
	$y_3$ cannot be in this region by the definition of lower triple.
	Hence, $y_2$ cannot be in this region. By the definition of layer decomposition, $B(C_{i+1,j}) \subseteq L_i$.
	So whenever the path from $y_1$ to $y_2$ leaves the cycle $B(C_{i+1,j})$ to its exterior, say at a vertex $q \in B(C_{i+1,j})$, then it has to return to $B(C_{i+1,j})$ at a vertex $q' \in N(q) \cap B(C_{i+1,j})$.
		The path from $y_1$ to $y_2$ will leave the region in some way.
	If it leaves along the circle $B(C_{i+1,j})$, either $f_3^1$ or $f_3^2$ must be on the path.
	If it leaves from inside, then $f_3$ must be on the path.
	Therefore, either $f_3^1$ or $f_3^2$ is in $S_i$ and this case also contradicts the fact that $P'$ avoids $S_i'$.
\end{proof}

%\vspace{.2cm}
\begin{lemma}[\cite{ABo}\footnote{The statement is not explicitly stated in \cite{ABo}, but can be obtained from Lemma~1 of \cite{ABo}. Though the triples are associated to non-vacuous layers in \cite{ABo}, the proof is essentially the same.}]\label{lemma: before count S_i'}
	$|S_i| \le 5(k_{i-1} + k_i + k_{i+1}) + 12c_{i+1}$.
\end{lemma}

%\vspace{.2cm}
\begin{lemma}\label{lemma: count S_i'}
	$|S_i'| \le 15(k_{i-1} + k_i + k_{i+1}) + 36c_{i+1}$.
\end{lemma}

\begin{proof}
Note that the number of vertices in generalized upper triples, generalized lower triples and generalized middle triples is at most three times that in upper triples, lower triples and middle triples. Thus the conclusion holds by Lemma \ref{lemma: before count S_i'}.
\end{proof}
%Hence, the total number of vertices in all generalized upper triples is bounded by $3k_{i-1} + 12(k_{i-1} + c_{i+1})$.

%Similarly, the total number of vertices in all generalized lower (resp. middle) triples is bounded by $3k_{i+1} + 12(k_{i+1} + c_{i+1})$ (resp. $3k_{i} + 12(k_{i} + c_{i+1})$).

%By the definition of $S_i'$, this proves our claim.

% \begin{proof}
% 	From \cite{ABo}, the number of pairs $(x, C_{i+1,j})$ such that $x \in D_{i-1}$ with neighbors in $B(C_{i+1,j})$ is bounded by $2(t_{i-1} + c_{i+1})$.
% 	The statement is not explicitly stated in \cite{ABo}, but can be obtained from the proof of Lemma~1 of \cite{ABo}.
% \end{proof}

%\vspace{.2cm}
\begin{lemma}\label{lemma: count c_i}
	$c_i \le k_{i+1} + k_{i+2} + k_{i+3} + k_{i+4}$.
\end{lemma}

\begin{proof}
	Recall that $c_i$ is the number of $3$-non-vacuous layer components in layer $L_i$, i.e., there is at least one edge between layer $L_{i+2}$ and $L_{i+3}$ contained within each such layer component.
	Such an edge can only be dominated by a vertex from layer $L_{i+1}, L_{i+2}, L_{i+3}$ or $L_{i+4}$. Moreover, each such edge must be dominated by a different vertex.
	Hence the result holds.
\end{proof}

% Combining Lemma~\ref{lemma: count S_i'} and Lemma~\ref{lemma: count c_i}, it results in:

%\vspace{.2cm}
\begin{prop}\label{prop: S_i' k-upper bound}
	$\sum_{i=1}^{r}|S_i'| \le 189k$, where $r$ is the number of layers of the graph.
\end{prop}
\begin{proof}
	This follows directly by Lemmas \ref{lemma: count S_i'}, \ref{lemma: count c_i} and the fact $\sum_{i=1}^{r}k_i \le k$.
\end{proof}

%\vspace{.2cm}
\begin{theorem}\label{thm: sqrt k separator}
	A planar graph with ve-domination number $k$ has treewidth of at most $ 18\sqrt{14} \sqrt{k} + 14 $.
\end{theorem}

\begin{proof}
	We consider the following five sets of vertices:
	$\mathcal{S}_t = S_t' \cup S_{t+5}' \cup S_{t+10}' \cup \ldots, t = 1,2,3,4,5$.
	Since $ \sum_{t=1}^{5} |\mathcal{S}_t| \le 189k$ (by Proposition~\ref{prop: S_i' k-upper bound}), one of these sets has size at most $\frac{189}{5}k$, say $\mathcal{S}_a$ with $a \in \{1,2,3,4,5\}$.
	
	 We now go through the sequence $S_{a}', S_{a+5}', S_{a+10}', \ldots$ and look for separators of size at most $s(k) = \alpha \sqrt{k}$ where $\alpha$ is a fixed constant.
	Due to the upper bound on the size of $\mathcal{S}_a$, such separators of size at most $s(k)$ must appear within every $n(k) \triangleq \frac{189k}{5\alpha \sqrt{k}} = \frac{189}{5\alpha} \sqrt{k}$ sets in the sequence.
		In this manner, we obtain a set of disjoint separators of size at most $s(k)$ each, such that any two consecutive separators from this set are at most $5n(k)$ layers apart.
	Clearly, the separators chosen in this way fulfill the requirements in Proposition~\ref{prop: treewidth layer decomposition}.

	Notice that the components cut out by chosen separators in this way each have at most $5(n(k) + 1)$ layers.
	Hence, their treewidth is at most $15(n(k) + 1) -1 $ due to Theorem~\ref{thm: outerplanarity treewidth}.
	
	By Proposition~\ref{prop: treewidth layer decomposition}, we can estimate the treewidth of the original graph $G$ with ve-domination number $k$:

	\begin{equation}
	\begin{aligned}
		tw(G) & \le 2 s(k) + 15(n(k) + 1) - 1 \\
			& = 2 \alpha \sqrt{k} + \frac{81 \cdot 7}{\alpha} \sqrt{k} + 14. \\
	\end{aligned}
	\end{equation}
	Let $\alpha = 9\sqrt{\frac{7}{2}}$, then we have $tw(G) \le 18\sqrt{14} \sqrt{k}+14$.
	This proves the theorem.
\end{proof}

Let $\gamma(G)$ and $\gamma_{ve}(G)$ be the domination number and ve-domination number of graph $G$ respectively.
\cite{ABo} proved $tw(G) \le O(\sqrt{\gamma(G)})$.
Now we have proved $tw(G) \le O(\sqrt{\gamma_{ve}(G)})$ which can derive the result of Alber~et~al~(up to constant factors) since $\gamma(G) \ge \gamma_{ve}(G)$.

Consider complete grid graph $G_n$ with $n^2$ vertices.
It is known that $tw(G_n) \ge n$ (see Corollary 89 of~\cite{TwofGridGraph}) and it is not hard to prove $\gamma_{ve}(G_n) = \Theta(n^2)$.
Therefore, $tw(G_n) = \Omega(\sqrt{\gamma_{ve}(G_n)})$ and $\lim \limits_{n\rightarrow +\infty}\gamma_{ve}(G_n) = +\infty$ which shows that our result in Theorem~\ref{thm: sqrt k separator} is optimal up to constant factors.

% ------------------------------------------------

\section{An algorithm on planar graphs}\label{sec: planar graph algorithm}

In Section~\ref{sec: planar treewidth ve-domination number}, we have $tw(G) \le O(\sqrt{\gamma_{ve}(G)})$ when $G$ is planar, and in Section~\ref{sec: tree width algorithm}, we have an algorithm for finding the ve-domination number in $O(5^{tw(G)}n)$ time.
Using the techniques in~\cite{ABo}, we can combine the above results to obtain an algorithm on planar graphs in $O(c^{\sqrt{k}}n)$ time solving the $k$-ve-dominating-set problem, that is, answering whether $\gamma_{ve}(G) \le k$.

\vskip.2cm
\begin{theorem}\label{thm: planar k ve dominating algorithm}
	There exists an algorithm that solves the $k$-ve-dominating-set problem on planar graphs in time $O(c^{\sqrt{k}}n)$, where $c=5^{18\sqrt{14}}$ and $n$ is the order of the input graph.
	Moreover, if $\gamma_{ve}(G) \le k$, a minimum-size ve-dominating set can be constructed within the same time.
\end{theorem}

\begin{proof} %Let $G$ be an $r$-outerplanar graph with layers $L_1, \ldots, L_r$. Let $L_i = \emptyset$ for all $i<0$ and $i > r$.
Our algorithm proceeds as follows:

\begin{enumerate}[label=Step \arabic*:]
	\item Embed the input planar graph $G(V,E)$ crossing-free into the plane. Determine the outerplanarity number $r$ of this embedding and get all layers $L_1,\ldots,L_r$. Let $L_i = \emptyset$ for all $i<0$ and $i > r$.
	\item For $\delta \in \{1,2,3,4,5\}$ and $i = 0, \ldots, \lfloor r/5 \rfloor -1$, find the minimum separator $\tilde{S}_{5i+\delta}$ which separates layers $L_{5i+\delta -1} $ and $L_{5i+\delta +4} $. Let $\tilde{s}_{5i+\delta} = |\tilde{S}_{5i+\delta}|$.
	\item Check whether there exists a $\delta \in \{1,2,3,4,5\}$ and an increasing sequence $(i_j)_{1 \le j \le t}$ of indices, such that
	\[ \tilde{s}_{5i_j + \delta} \le s(k) = 9\sqrt{ \frac{7}{2}}\sqrt{k}, ~~~~~~~\mbox{ for all $j=1,2,\ldots,t$ ~~~and } \]
	\[ |i_{j+1} - i_j |  \le n(k) =  \frac{3\sqrt{14}}{5}\sqrt{k}, \mbox{ for all $j=1,2,\ldots,t$ where $i_{t+1} = r+1$.} \]
	If the answer is ``no'', then there is no $k$-ve-dominating set.
	\item Consider the separators $\mathcal{S}_j = \tilde{S}_{5i_j +\delta}$ for $j=1,2,\ldots, t$ and let $\mathcal{S}_j = \emptyset$ for all other $j$.
	% For each $j=0,\ldots, t$, let $G_j$ be the subgraph cut out by separators $S_j$ and $S_{j+1}$, or, more precisely, let
	% \[
	% 	G_j = G - \left( \bigcup_{l=(3i_j+\delta) - 1}^{3i_{j+1}+\delta +1} L_{l} \backslash ( \mathcal{S}_j \cup \mathcal{S}_{j+1} ) \right)	
	% \]
	Let $\{G_i\}_{1 \le i \le m}$ be the connected components in $G[V-\bigcup_{1 \le j \le t}\mathcal{S}_j]$.
	Note that $G_i$ is at most $5(n(k) + 1)$-outerplanar for all $1 \le i \le m$.
	\item Construct tree-decompositions $\mathcal{T}_i$ for $G_i (i=1,\ldots,m)$ with $O(n)$ nodes and width at most $15n(k)+14$ each.
	\item Construct tree-decomposition $\mathcal{T}$ of $G$ with $O(|V(G_i)|)$ nodes using $\mathcal{T}_i, 1 \le i \le m$ and $\mathcal{S}_j, 1 \le j \le t$.
	\item Solve the ve-dominating set problem for $G$ with tree-decomposition $\mathcal{T}$ using the algorithm in Theorem~\ref{thm: main thm: poly algorithm of ve-dominating set}.
\end{enumerate}

Then we go into details of each step.
Step 1 can be solved in linear time~\cite{embedding_chiba1985linear} and Step 2 can be solved with well-known techniques based on maximum flow~\cite{embedding_chiba1985linear}.
In Step 3, by the consideration in Theorem~\ref{thm: sqrt k separator}, if $\gamma_{ve}(G) \le k$, such $\delta$ and sequence must exist. Hence, when the answer of Step 3 is ``no'', $\gamma_{ve}(G) > k$ must hold.
Step 5 is justified by Theorem~\ref{thm: outerplanarity treewidth} and can be solved in $O(n\sqrt{k})$ time by Theorem~14 of~\cite{ABo}.
% Theorem~\ref{thm: algorithm outerplanar treewidth}

% %\vspace{.2cm}
% \begin{theorem}[\cite{ABo}]\label{thm: algorithm outerplanar treewidth}
% 	Let an $r$-outerplanar graph $G(V,E)$ be given together with an $r$-outerplanar embedding. Then a tree-decomposition with width at most $3r-1$ and with $O(n)$ nodes can be found in $O(rn)$ time.
% \end{theorem}

In Step 6, we can construct tree-decomposition $\mathcal{T}$ in the following way (see Figure~\ref{fig: construction step 6}):
\begin{enumerate}[label=(\arabic*)]
	\item We say a $G_i$ connects a separator $\mathcal{S}_j$ if there is an edge between $G_i$ and $\mathcal{S}_j$ in $G$.
	By default, when $\mathcal{S}_j$ is an empty set, we also say $G_i$ connects $\mathcal{S}_j$.
	By the way we choose separators, each $G_i$ connects at most two separators (except empty sets), and, if so, the two separators are consecutive, that is, $\mathcal{S}_{p}$ and $\mathcal{S}_{p+1}$ for some $p$.
	For each separator $\mathcal{S}_j, j = 0, \ldots, t+1$, create a node containing all vertices of $\mathcal{S}_j$ which is still denoted by $\mathcal{S}_j$.
	\item For each fixed $0 \le j \le t$, assume 
	\begin{equation*}
	\begin{aligned}
		&\{G_{i_1}, \ldots, G_{i_k}\} \\ 
		=& \{G_i \mid \mbox{$G_i$ connects $\mathcal{S}_j$ and $\mathcal{S}_{j+1}$ or $G_i$ only connects $\mathcal{S}_{j+1}$ (except empty sets)}\}.
	\end{aligned}
	\end{equation*}
	In the tree-decomposition $\mathcal{T}$, sequentially connect $\mathcal{S}_j, \mathcal{T}_{i_1}, \ldots, \mathcal{T}_{i_k}, \mathcal{S}_{j+1}$.
	\item For each $0 \le j \le t$ and $\{G_{i_1}, \ldots, G_{i_k}\}$ is defined as before.
	If $N$ is a node of $\mathcal{T}_{i_l}$ for some $1 \le l \le k$, replace $N$ by $N' = N \cup \mathcal{S}_{j} \cup \mathcal{S}_{j+1}$.
\end{enumerate}

It is easy to verify $\mathcal{T}$ is indeed a tree-decomposition of $G$ with width at most $15n(k) + 14 + 2s(k) = 18\sqrt{14}\sqrt{k} + 14$.

\begin{figure}[t]
	\centering         %使图片居中放置
	\includegraphics[width=0.9\linewidth]{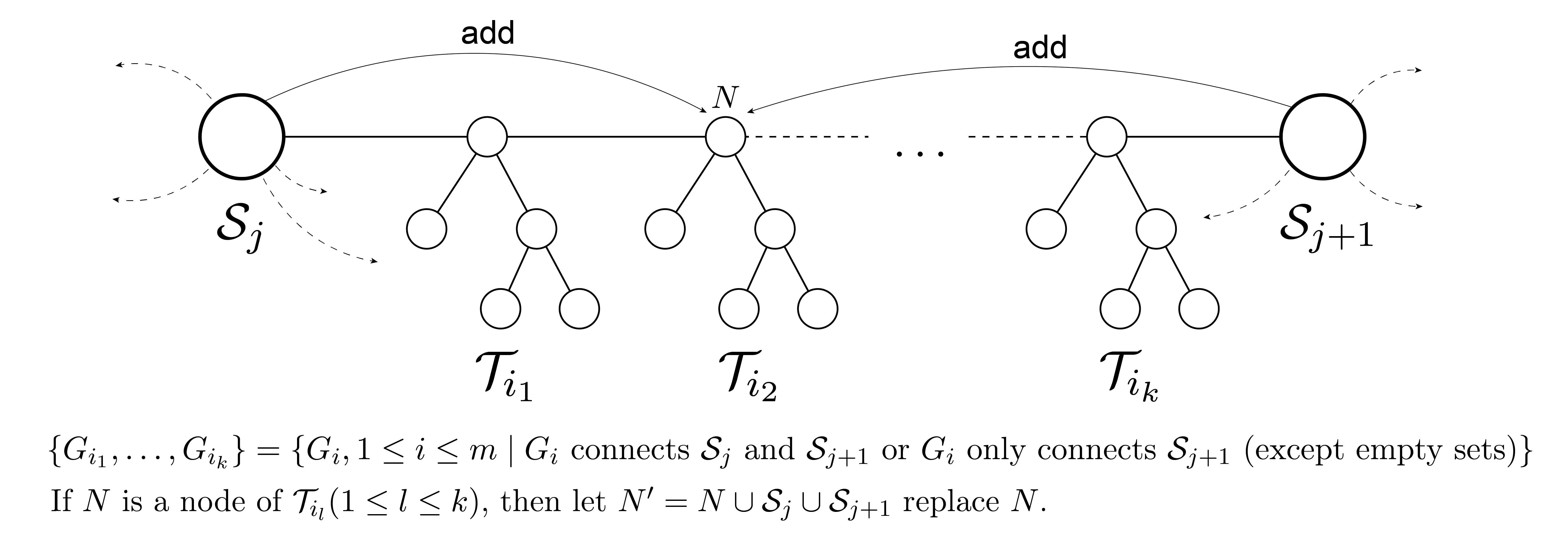}
	\caption{Construction of Step 6.}
	\label{fig: construction step 6}
\end{figure}

Obviously, the running time bottleneck is Step 7 and the whole algorithm takes $O(c^{\sqrt{k}}n)$ time, where $c = 5^{18\sqrt{14}}$.
\end{proof}

% ------------------------ Conclusions ----------------------
\section{Conclusions}\label{sec: conclusions}

We first reviewed the research status of ve-dominating set and introduced treewidth.
Then, we established a polynomial-time algorithm for calculating the ve-domination number on graphs with bounded treewidth.

We show that for the treewidth of planar graphs, we have $tw(G) \le 18\sqrt{14} \sqrt{\gamma_{ve}(G)} + 14$ where $\gamma_{ve}(G)$ is the ve-domination number of $G$ and the result is optimal up to constant factors.
The constant $18\sqrt{14}$ is huge here and we believe it can be improved by estimating the upper bound more finely.

As treewidth is an important parameter of graphs, there are other commonly used graph parameters like splitwidth and cliquewidth.
It is worth trying to solve the ve-domination problem on graphs with those bounded parameters.
Moreover, there are other variations of domination number.
It would be interesting to investigate the relationship between treewidth and extended domination number on planar graphs in similar ways.

We also establish an algorithm solving the $k$-ve-domination problem on planar graphs in $O(c^{\sqrt{k}}n)$ time.
In the proof of Theorem~\ref{thm: planar k ve dominating algorithm}, we actually present a way to build a tree-decomposition of a planar graph $G$ with width at most $O(\sqrt{k})$ if it is known that $\gamma_{ve}(G) \le k$.
It may be useful to design other algorithms on planar graphs.

\section*{Acknowledgement}
The research of Lu is supported by the National Natural Science Foundation of China (Grant 12571372).

\bibliographystyle{wyc4}
\bibliography{ref.bib}

@book{nice_tree_decomposition_kloks_book1994,
  title={Treewidth: computations and approximations},
  author={Kloks, Ton},
  year={1994},
  publisher={Springer}
}

@book{vedomination_introduce,
  title={Theoretical and algorithmic results on domination and connectivity (Nordhaus-Gaddum, Gallai type results, max-min relationships, linear time, series-parallel)},
  author={Peters Jr, Kenneth W},
  year={1986},
  publisher={Clemson University}
}

@phdthesis{vedomination_introduce2,
  title={Vertex-edge and edge-vertex parameters in graphs},
  author={Lewis, Jason Robert},
  year={2007},
  school={Clemson University}
}

@article{paul2022,
  title={Results on vertex-edge and independent vertex-edge domination},
  author={Paul, Subhabrata and Ranjan, Keshav},
  journal={Journal of Combinatorial Optimization},
  volume={44},
  number={1},
  pages={303--330},
  year={2022},
  publisher={Springer}
}

@article{krishnakumari2014bounds,
  title={Bounds on the vertex--edge domination number of a tree},
  author={Krishnakumari, Balakrishna and Venkatakrishnan, Yanamandram B and Krzywkowski, Marcin},
  journal={Comptes rendus mathematique},
  volume={352},
  number={5},
  pages={363--366},
  year={2014},
  publisher={Elsevier}
}

@article{treewidth_sample1_Arnborg1991,
  title={Easy problems for tree-decomposable graphs},
  author={Arnborg, Stefan and Lagergren, Jens and Seese, Detlef},
  journal={Journal of Algorithms},
  volume={12},
  number={2},
  pages={308--340},
  year={1991},
  publisher={Elsevier}
}

@inproceedings{dominationset_treewidth_alber2002,
  title={Improved tree decomposition based algorithms for domination-like problems},
  author={Alber, Jochen and Niedermeier, Rolf},
  booktitle={Latin American Symposium on Theoretical Informatics},
  pages={613--627},
  year={2002},
  organization={Springer}
}

@article{VCP3_treewidth_bai2019,
  title={An improved algorithm for the vertex cover $ {P}_3 $ problem on graphs of bounded treewidth},
  author={Bai, Zongwen and Tu, Jianhua and Shi, Yongtang},
  journal={Discrete Mathematics \& Theoretical Computer Science},
  volume={21},
  year={2019},
  publisher={Episciences. org}
}

@article{ABo,
  title={Fixed parameter algorithms for dominating set and related problems on planar graphs},
  author={Alber and Bodlaender and Fernau and Kloks and Niedermeier},
  journal={Algorithmica},
  volume={33},
  pages={461--493},
  year={2002},
  publisher={Springer}
}

@article{TwofGridGraph,
  title={A partial k-arboretum of graphs with bounded treewidth},
  author={Bodlaender, Hans L},
  journal={Theoretical computer science},
  volume={209},
  number={1-2},
  pages={1--45},
  year={1998},
  publisher={Elsevier}
}

@inproceedings{PlanarityTreewidth,
  title={Treewidth: Algorithmic techniques and results},
  author={Bodlaender, Hans L},
  booktitle={International Symposium on Mathematical Foundations of Computer Science},
  pages={19--36},
  year={1997},
  organization={Springer}
}

@article{Baker,
  title={Approximation algorithms for NP-complete problems on planar graphs},
  author={Baker, Brenda S},
  journal={Journal of the ACM (JACM)},
  volume={41},
  number={1},
  pages={153--180},
  year={1994},
  publisher={ACM New York, NY, USA}
}

@inproceedings{peng2007roman,
  title={Roman domination on graphs of bounded treewidth},
  author={Peng, Sheng-Lung and Tsai, Yuan-Hsiang},
  booktitle={Proceedings of the 24th Workshop on Combinatorial Mathematics and Computation Theory},
  pages={128--131},
  year={2007}
}

@article{boutrig2016vertex,
  title={Vertex-edge domination in graphs},
  author={Boutrig, Razika and Chellali, Mustapha and Haynes, Teresa W and Hedetniemi, Stephen T},
  journal={Aequationes mathematicae},
  volume={90},
  pages={355--366},
  year={2016},
  publisher={Springer}
}

@article{zylinski2019vertex,
  title={Vertex-edge domination in graphs},
  author={{\.Z}yli{\'n}ski, Pawe{\l}},
  journal={Aequationes mathematicae},
  volume={93},
  number={4},
  pages={735--742},
  year={2019},
  publisher={Springer}
}

@article{embedding_chiba1985linear,
  title={A linear algorithm for embedding planar graphs using PQ-trees},
  author={Chiba, Norishige and Nishizeki, Takao and Abe, Shigenobu and Ozawa, Takao},
  journal={Journal of computer and system sciences},
  volume={30},
  number={1},
  pages={54--76},
  year={1985},
  publisher={Elsevier}
}

\end{document}